\documentclass[11pt]{amsart}
\usepackage{amsmath,mathrsfs,cite}
\usepackage[colorlinks=true,urlcolor=blue,
citecolor=red,linkcolor=blue,linktocpage,pdfpagelabels,bookmarksnumbered,bookmarksopen]{hyperref}
\usepackage{color,mathrsfs}
\usepackage[english]{babel}
\usepackage[left=2.7cm,right=2.7cm,top=2.84cm,bottom=2.84cm]{geometry}

\numberwithin{equation}{section}
\newtheorem{theorem}{Theorem}[section]
\newtheorem{proposition}[theorem]{Proposition}
\newtheorem{lemma}[theorem]{Lemma}
\newtheorem{remark}[theorem]{Remark}

\theoremstyle{definition}

\newcommand{\R}{{\mathbb R}}

\newcommand{\T}{{\mathbb T}}
\newcommand{\dvg}{{\rm div}}

\newcommand{\eps}{\varepsilon}

\def\de{\partial}

\title[Explosive solutions for a class of 
quasi-linear PDEs]{On explosive solutions for a class of \\ quasi-linear elliptic equations}

\author[F.~Gladiali]{Francesca Gladiali}
\author[M.~Squassina]{Marco Squassina}

\address{Universit\`a degli Studi di Sassari
\newline\indent
Via Piandanna 4, I-07100 Sassari, Italy}
\email{fgladiali@uniss.it}

\address{Universit\`a degli Studi di Verona
\newline\indent
Strada Le Grazie 15, I-37134 Verona, Italy}
\email{marco.squassina@univr.it}

\thanks{The authors were partially supported by the PRIN projects {\em ``Variational Methods and Nonlinear Differential Equations''} and
 {\em ``Variational and Topological Methods in the Study of Nonlinear Phenomena''} respectively}


\begin{document}

\begin{abstract}
We study existence, uniqueness, multiplicity and symmetry of large solutions for a class of quasi-linear elliptic equations.
Furthermore, we characterize the boundary blow-up rate of solutions, including the case where the contribution of boundary curvature appears.
\end{abstract}

\maketitle

\bigskip
\setcounter{tocdepth}{1}
\begin{center}
\begin{minipage}{11cm}
\footnotesize
\tableofcontents
\end{minipage}
\end{center}



\section{Introduction and results}

The study of explosive solutions of elliptic equations goes back 
to 1916 by Bieberbach \cite{bieber} for the problem $\Delta u = e^{u}$ on a bounded two dimensional domain,
arising in Riemannian geometry as related to exponential metrics with 
constant Gaussian negative curvature. The result was then extended to three dimensional domains
by Rademacher \cite{rademacher} in 1943. Large solutions of more general elliptic 
equations $\Delta u = f (u)$ in smooth bounded domains $\Omega$ of $\R^N$ were originally studied
by Keller \cite{keller} and Osserman \cite{osserman} around 1957, and subsequently 
refined in a series of more recent contributions, see \cite{AR,BM,BM2,cirad1,cirad2,CD,ddgr,farinahb,gherad-book,LM,marcvero,PV}
and references therein. 
The aim of this paper is to study existence, uniqueness, symmetry as well as 
asymptotic behavior on $\partial\Omega$ for the quasi-linear problem
\begin{equation}
\label{prob}
\begin{cases}
\,\dvg(a(u)Du)=\frac{a'(u)}{2}|Du|^2+f(u)   & \text{in $\Omega$,} \\
\noalign{\vskip2pt}
\,\text{$u(x)\to+\infty$\quad as ${\rm d}(x,\partial\Omega)\to 0$,}    &
\end{cases}
\end{equation}
where $\Omega$ is a bounded smooth domain of $\R^N$, $N\geq 1$, and ${\rm d}(x,\partial\Omega)$
is the distance of $x$ from the boundary of $\Omega$. 
Here and in the following $a:\R\to\R^+$ is a $C^1$ function 
such that there exists $\nu>0$ with $a(s)\geq \nu$ for any $s\in \R$, and $f:\R\to\R$ is a $C^1$ function.
In problem~\eqref{prob}, the terms depending upon $a$ are formally associated with the functional $\int_\Omega a(u)|Du|^2$
and the problem can be thought as related to the study of blow-up solutions in presence of 
a Riemannian metric tensor depending upon the unknown $u$ itself,
see e.g.\ \cite{scho-uhl,uhlen} for more details.
We shall cover the situations where $a$ and $f$ have an exponential, polynomial or logarithmic type growth at infinity. In the semi-linear
case $a\equiv 1$, typical situations where the exponential nonlinearity appears is the Liouville \cite{liouv} equation
$\Delta u=4 e^{2u}$ in $\Omega\subset\R^2$, while for a typical polynomial growth one can think 
to the Loewner-Nirenberg \cite{loenir} equation $\Delta u=3u^{5}$ in $\Omega\subset\R^3$.
Logarithmic type nonlinearities usually appear in theories of quantum gravity \cite{log1} and in particular in the
framework of nonlinear Schr\"odinger equations \cite{log2}.
The function $a$ can be regarded as responsible for the diffusion effects while, on the contrary, $f$ can be
considered as playing the r\v ole of an external source. Roughly speaking, in some sense, 
$a$ is competing with $f$ for the existence and nonexistence of solutions
to~\eqref{prob} and the asymptotic behavior of $a(s)$ and $f(s)$ as $s\to+\infty$ determines the blow-up rate
of $u(x)$ as $x$ approaches the boundary of $\Omega$. For the literature on these type of quasi-linear operators
in frameworks different from that of large solutions, we refer the reader to \cite{squassmono} and the reference therein.
In order to give precise characterization of existence
and explosion rate, we shall convert the quasi-linear problem \eqref{prob} into a corresponding semi-linear
problem through a change of variable procedure involving the globally defined Cauchy problem
\begin{equation}
	\label{cauchy}
g'=\frac{1}{\sqrt{a\circ g}},\qquad g(0)=0.
\end{equation}
The precise knowledge of the asymptotic behavior of the solution $g$ of~\eqref{cauchy} as $s\to+\infty$ depending of the asymptotics 
of the function $a$ will be crucial in studying the qualitative properties of the solutions to~\eqref{prob}. 
We shall obtain for \eqref{prob} existence, nonexistence, uniqueness and multiplicity 
results in arbitrary smooth bounded domains, uniqueness 
and symmetry results when the problem is set in the ball and, finally, results about the blow-up 
rate of the solution with or without the second order contribution of the local curvature of the boundary $\partial\Omega$.
For instance, if $a(s)\sim a_\infty s^k$ as $s\to+\infty$ and $f(s)\sim f_\infty s^p$ as $s\to+\infty$ with $p>2k+3$, then a solution
to~\eqref{prob} always exists and any solution satisfies, as $x$ approaches $\partial\Omega$,
$$
u(x)=\frac{\Gamma}{({\rm d}(x,\partial\Omega))^{\frac{2}{p-k-1}}}(1+o(1)),\quad\,\,\,
\Gamma=\Big[\frac{p-k-1}{\sqrt{2(p+1)}} \frac{\sqrt{f_{\infty}}}{\sqrt {a_{\infty}}}\Big]^{\frac{2}{k+1-p}}.
$$
If instead $k+1<p\leq 2k+3$, then we have
$$
u(x)=\Gamma\frac{1}{({\rm d}(x,\partial\Omega))^{\frac{2}{p-k-1}}}(1+o(1))
+\Gamma'\frac{{\mathcal H}(\sigma(x)) }{({\rm d}(x,\partial\Omega))^{\frac{3+k-p}{p-k-1}}}(1+o(1)),
\quad\,\,\, \Gamma '=\frac{2(N-1)}{p+k+3}\Gamma,
$$
for $x$ approaching $\partial\Omega$,
being $\sigma(x)$ the orthogonal projection on $\partial\Omega$ of a $x\in\Omega$ and 
denoting ${\mathcal H}$ the mean curvature of the boundary $\partial\Omega$.
\vskip2pt
\noindent
In this paper we shall restrict the attention on the study of explosive solutions in smooth and bounded domains. Concerning
the study of large solutions of quasi-linear equations (including non-degenerate and non-autonomous problems)
on the entire space, a vast recent literature currently exists on the subject. We refer
the reader to the contributions \cite{dupai,dupghegouwar,farser1,farser2,farinahb,roberta1,roberta2,filpucrig,puccser}
of (in alphabetical order) Dupaigne, Farina, Filippucci, Pucci, Rigoli and Serrin and the references therein.
\vskip4pt
\noindent
Concerning the existence of solutions to \eqref{prob}, we have the following

\begin{theorem}[\textcolor{blue}{Existence of solutions}]
\label{main-ex}
The following facts hold:
\begin{enumerate}
\item Assume that there exist $k>0$, $\beta>0$, $a_\infty>0$ and $f_\infty>0$ such that
\begin{equation}
		\label{caso:potexp}
\lim_{s\to +\infty}\frac{a(s)} {s^k}=a_\infty,\qquad
\lim_{s\to +\infty}\frac{f(s)} {e^{2\beta s}}=f_\infty.
\end{equation}
Then~\eqref{prob} admits a solution.
\vskip3pt
\item Assume that there exist $k>0$, $p>0$, $a_\infty>0$ and $f_\infty>0$ such that
\begin{equation}
	\label{caso:potpot}
\lim_{s\to +\infty}\frac{a(s)} {s^k}=a_\infty,\qquad
\lim_{s\to +\infty}\frac{f(s)} {s^p}=f_\infty.
\end{equation}
Then~\eqref{prob} admits a solution if and only if $p>k+1$.
\vskip3pt
\item Assume that there exist $k>0$, $\beta>0$, $a_\infty>0$ and $f_\infty>0$ such that
\begin{equation}
		\label{caso:potlog}
\lim_{s\to +\infty}\frac{a(s)} {s^k}=a_\infty,\qquad
\lim_{s\to +\infty}\frac{f(s)} {(\log s)^\beta}=f_\infty.
\end{equation}
Then~\eqref{prob} admits no solution.
\vskip3pt
\item Assume that there exist $\gamma>0$, $\beta>0$, $a_\infty>0$ and $f_\infty>0$ such that
\begin{equation}
		\label{caso:expexp}
\lim_{s\to +\infty}\frac{a(s)} {e^{2\gamma s}}=a_\infty,\qquad
\lim_{s\to +\infty}\frac{f(s)} {e^{2\beta s}}=f_\infty.
\end{equation}
Then~\eqref{prob} admits a solution if and only if $\beta>\gamma$.
\vskip3pt
\item Assume that there exist $\gamma>0$, $p>0$, $a_\infty>0$ and $f_\infty>0$ such that
\begin{equation*}
\lim_{s\to +\infty}\frac{a(s)} {e^{2\gamma s}}=a_\infty,\quad
\lim_{s\to +\infty}\frac{f(s)} {s^p}=f_\infty.
\end{equation*}
Then~\eqref{prob} admits no solution.
\vskip3pt
\item Assume that there exist $\gamma>0$, $\beta>0$, $a_\infty>0$ and $f_\infty>0$ such that
\begin{equation*}
\lim_{s\to +\infty}\frac{a(s)} {e^{2\gamma s}}=a_\infty,\qquad
\lim_{s\to +\infty}\frac{f(s)} {(\log s)^\beta}=f_\infty.
\end{equation*}
Then~\eqref{prob} admits no solution.
\vskip3pt
\item Assume that there exist $\gamma>0$, $\beta>0$, $a_\infty>0$ and $f_\infty>0$ such that
\begin{equation}
		\label{caso:logexp}
\lim_{s\to +\infty}\frac{a(s)} {(\log s)^{2\gamma}}=a_\infty,\qquad	
\lim_{s\to +\infty}\frac{f(s)} {e^{2\beta s}}=f_\infty.
\end{equation}
Then~\eqref{prob} admits a solution.
\vskip3pt
\item Assume that there exist $\gamma>0$, $p>0$, $a_\infty>0$ and $f_\infty>0$ such that
\begin{equation}
		\label{caso:logpot}
\lim_{s\to +\infty}\frac{a(s)} {(\log s)^{2\gamma}}=a_\infty,\qquad
\lim_{s\to +\infty}\frac{f(s)} {s^p}=f_\infty.
\end{equation}
Then~\eqref{prob} admits a solution if and only if $p>1$.
\vskip3pt
\item Assume that there exist $\gamma>0$, $\beta>0$, $a_\infty>0$ and $f_\infty>0$ such that
\begin{equation*}
\lim_{s\to +\infty}\frac{a(s)} {(\log s)^{2\gamma}}=a_\infty, \qquad	
\lim_{s\to +\infty}\frac{f(s)} {(\log s)^\beta}=f_\infty.
\end{equation*}
Then~\eqref{prob} admits no solution.
\end{enumerate}
\end{theorem}
\noindent
Theorem~\ref{main-ex} follows by Propositions~\ref{potexp}, \ref{potpot}, \ref{pollog}, 
\ref{expexp}, \ref{exppot}, \ref{explog}, \ref{logexp}, \ref{logpot} and \ref{loglog}.
\vskip4pt

\noindent
Concerning the uniqueness  of solutions, we have the following

\begin{theorem}[\textcolor{blue}{Uniqueness of solutions}]
	\label{uniq}
Suppose 
that
\begin{equation}\label{u1}
2 f'(s)a(s)-f(s)a'(s)\geq 0\quad \text{ for $s\geq 0$}, \quad f(s)=0\quad\text{for $s\leq 0$}, \quad f(s)>0\quad\text{ for $s>0$},
\end{equation}
and that 
\begin{equation}\label{u2}
\lim_{s\to +\infty}\frac{\big[2 f'(s)a(s)-f(s)a'(s) \big]g^{-1}(s)}{2 (a(s))^{\frac 32}f(s)}>1,
\end{equation}
where $g$ is as defined in~\eqref{cauchy}.
Then~\eqref{prob} admits a unique solution, which is positive. \\
Moreover, assume that $a$ and $f$ satisfy one of the existence conditions of Theorem~\ref{main-ex} and $\partial \Omega$ is of class $C^3$ and its mean curvature is nonnegative.
Then if \eqref{u1} is satisfied and if 
\begin{equation}\label{u1b}
\text{there exists $R>0$ such that:$\qquad  \Big (\frac {f(g(s))}{\sqrt{a(g(s))}}\Big)^{\frac 12}$ is convex in $(R,+\infty)$,}
\end{equation} 
then~\eqref{prob} admits a unique solution, which is positive. \\
\noindent 
Consider now also problem~\eqref{prob} set in the unit ball $B_1(0)$
\begin{equation}
\label{prob-ball}
\begin{cases}
\,\dvg(a(u)Du)=\frac{a'(u)}{2}|Du|^2+f(u)   & \text{in $B_1(0)$,} \\
\noalign{\vskip2pt}
\,\text{$u(x)\to+\infty$\quad as ${\rm d}(x,\partial B_1(0))\to 0$.}    &
\end{cases}
\end{equation}
Let $\lambda_1$ be the first eigenvalue of $-\Delta$ in $B_1(0)$ with Dirichlet boundary conditions.
Assume $a$ and $f$ satisfy one of the existence conditions of Theorem~\ref{main-ex} and, in addition, that
\begin{equation}
	\label{uniqcond}
2 f'(s)a(s)-f(s)a'(s)+2\lambda_1 a^2(s)\geq 0,\qquad \text{for all $s\in\R$.}
\end{equation}
Then~\eqref{prob-ball} admits a unique solution. 
\end{theorem}
\noindent
The proof of Theorem~\ref{uniq} follows by Proposition~\ref{uniglob}.
\vskip4pt
\noindent
Concerning the multiplicity of solutions, we have the following

\begin{theorem}[\textcolor{blue}{Nonuniqueness of solutions}]
	\label{molt}
Assume that the functions $a$ and $f$ satisfy one of the existence conditions of Theorem~\ref{main-ex}.
Let $\Omega$ be bounded, convex, $C^2$, $f(0)=0$ and 
\begin{equation}
	\label{segnomonot}
	\text{there exists $R\geq 0$ such that:}\qquad
f|_{(R,+\infty)}>0,\qquad
(2f'a-fa')|_{(R,+\infty)}\geq 0.
\end{equation}
Assume that there exists $1<q<\frac{N+2}{N-2}$ if $N\geq 3$, $q>1$ if $N=1,2$ such that  
\begin{equation*}
0<\lim_{s\to-\infty}\frac{f(s)}{\sqrt{a(s)}|g^{-1}(s)|^q}<+\infty.
\end{equation*}
Then \eqref{prob} admits two solutions, one positive and one sign-changing. In particular, if there exist $k>0$ and $p_-,p_+>1$ such that
\begin{align}
\label{specialnonu1}
&  0<\lim_{s\to +\infty}\frac{a'(s)}{s^{k-1}}<+\infty,\quad\,\,\, \text{$a(-s)=a(s)$,\,\,\, for all $s\in\R$,\,\,\, 
} \\
\label{specialnonu2}
&  0<\lim_{s\to-\infty}\frac{f(s)}{|s|^{p_-}}<+\infty,\quad\,\,\, \text{$k+1<p_-<\frac{k}{2}+\frac{k+2}{2}\frac{N+2}{N-2}$\,\,\, for $N\geq 3$,} \\
\label{specialnonu3}
&  0<\lim_{s\to-\infty}\frac{f(s)}{|s|^{p_-}}<+\infty,\quad\,\,\, \text{$p_->k+1$\,\,\, for $N=1,2$,} \\
\label{specialnonu4}
&  0<\lim_{s\to+\infty}\frac{f'(s)}{s^{p_+-1}}<+\infty,\quad\,\,\, \text{$p_+>k+1$\,\,\, for $N\geq 1$,} 
\end{align}
then~\eqref{prob} admits two solutions, one positive and one sign-changing.
\end{theorem}
\noindent
The proof of Theorem~\ref{molt} follows by Propositions~\ref{nonuniq} and~\ref{nonuniq-special}.
\vskip4pt

\noindent
Concerning the symmetry of solutions to~\eqref{prob-ball}, we have the following

\begin{theorem}[\textcolor{blue}{Symmetry of solutions}]
\label{main-sim} Let $a$ and $f$ be of class $C^2(\R)$.
The following facts hold:
\begin{enumerate}
\item Assume that there exist $k>0$, $\beta>0$, $a_\infty>0$ and $f_\infty>0$ such that
\begin{align}
	\label{ader-asi}
&	
\lim_{s\to +\infty}\frac{a'(s)} {s^{k-1}}=k a_\infty ,\quad	
\lim_{s\to +\infty}\frac{a''(s)} {s^{k-2}}=k(k-1)a_\infty ,   \\
	\label{fder-asi}
&	
\lim_{s\to +\infty}\frac{f''(s)} {e^{2\beta s}}=4\beta^2 f_\infty,
\end{align}
(only the right limit in \eqref{ader-asi} for $k>1$).
Then any solution to~\eqref{prob-ball} is radially symmetric and increasing.
\vskip3pt
\item Assume that there exist $k>0$, $p>1$, $a_\infty>0$ and $f_\infty>0$
such that~\eqref{ader-asi} hold and 
\begin{equation}
	\label{fpotder-asi}
\lim_{s\to +\infty}\frac{f''(s)} {s^{p-2}}=p(p-1) f_\infty.
\end{equation}
Then, if $p>k+1$, any solution to~\eqref{prob-ball} is radially symmetric and increasing.
\vskip3pt
\item Assume that there exist $\gamma>0$, $\beta>0$, $a_\infty>0$ and $f_\infty>0$ such that~\eqref{fder-asi} holds and
\begin{equation*}
\lim_{s\to +\infty}\frac{a''(s)} {e^{2\gamma s}}=4\gamma^2 a_\infty.
\end{equation*}
Then, if $\beta>\gamma$, any solution to~\eqref{prob-ball} is radially symmetric and increasing.
\vskip3pt
\item Assume that there exist $\gamma>0$, $\beta>0$, $a_\infty>0$ and $f_\infty>0$ 
such that~\eqref{fder-asi} holds and
\begin{equation}\label{ip-log}
\lim_{s\to +\infty}\frac{a'(s)s} {(\log s)^{2\gamma-1}}=2\gamma a_\infty,\,\,\,
\lim_{s\to +\infty}\frac{a''(s)s^2} {(\log s)^{2\gamma-1}}=-2\gamma a_\infty.
\end{equation}
Then, any solution to~\eqref{prob-ball} is radially symmetric and increasing.
\vskip3pt
\item Assume  that~\eqref{fpotder-asi} and \eqref{ip-log} hold.
Then, if $p>1$, any solution to~\eqref{prob-ball} is radially symmetric and increasing.
\end{enumerate}
\end{theorem}
\noindent
The proof of Theorem~\ref{main-sim} follows by Propositions~\ref{potexp-rad}, \ref{potpot-rad}, 
\ref{expexp-rad}, \ref{logexp-rad} and \ref{logpol-rad}.
\vskip4pt
\noindent
Concerning the blow-up behavior of solution, we have two results. The first is the following

\begin{theorem}[\textcolor{blue}{Boundary behavior I}]
	\label{blowI}
Let $\Omega$ be a bounded domain of $\R^N$ satisfying an inner and an outer sphere
condition on $\partial\Omega$. Let $\eta$ denote the unique solution to 
problem
\begin{equation}
	\label{ODEeta}
\eta'=-\sqrt{2F\circ g\circ\eta},\qquad \lim_{t\to 0^+}\eta(t)=+\infty.
\end{equation}
Then any solution $u\in C^2(\Omega)$ to~\eqref{prob} satisfies
\begin{equation*}
u(x)=g\circ\eta({\rm d}(x,\partial\Omega))+o(1),
\end{equation*}
whenever ${\rm d}(x,\partial\Omega)$ goes to zero, provided that one of the following situations occurs:
\begin{enumerate}
	\item  Conditions~\eqref{caso:potexp} hold.
	\item  Conditions~\eqref{caso:potpot} hold with $p>2k+3$.
	\item  Conditions~\eqref{caso:expexp} with $\beta>2\gamma$.
	\item  Conditions~\eqref{caso:logexp} hold.
	\item  Conditions~\eqref{caso:logpot} hold with $p>3$.
\end{enumerate}
\end{theorem}

\noindent
The proof of Theorem~\ref{blowI} follows by Propositions~\ref{blow-polpol}, \ref{blow-polexp}, 
\ref{blow-expexp}, \ref{blow-logpot}, \ref{blow-logexp}.
\vskip4pt
\noindent
In general, in addition to the blow-up term $g\circ\eta({\rm d}(x,\partial\Omega))$, the expansion of a large solution $u$ could
contain other blow-up terms, one of them typically depends upon the local mean curvature of the boundary. We will study this
in a particular, but meaningful, situation.

\noindent
For $p>k+1$, in the framework of \eqref{caso:potpot}, let us now introduce the positive constants
\begin{equation}
	\label{defGamma}
\Gamma:=\left[\frac{p-k-1}{\sqrt{2(p+1)}} \frac{\sqrt{f_\infty}}{\sqrt {a_\infty}}\right]^{\frac{2}{k+1-p}} \,\,\quad\text{and}\quad\,\,
\Gamma ':=\frac {2(N-1)}{p+k+3}\Gamma.
\end{equation}
When $a$ and $f$ behave like powers at infinity, we have the following blow-up characterization

\begin{theorem}[\textcolor{blue}{Boundary behavior II}]
	\label{blowII}
	Let $\Omega$ be a bounded domain of $\R^N$ of class $C^4$
	and assume that \eqref{caso:potpot} hold with $p>k+1$. 
	If conditons \eqref{segnomonot} hold with $R=0$ and $\eta$ is as in \eqref{ODEeta}, let us set
	\begin{align*}
	& J(t):=\frac{N-1}{2}\int_{0}^t \frac{\int_{0}^{\eta(s)}\sqrt{2F(g(\sigma))}d\sigma}{F(g(\eta(s)))}ds,\qquad t>0, \\
	& \T(x):=\frac{({\rm d}(x,\partial\Omega))^{\frac{k+2}{k+1-p}}}{\min\{u^{k/2}(x)),
	({\rm d}(x,\partial\Omega)-{\mathcal H}(\sigma(x))J({\rm d}(x,\partial\Omega)))^{k/(k+1-p)}\}},
	\qquad x\in\Omega,
\end{align*}
where $\sigma(x)$ denotes 
the projection on $\partial\Omega$ of a $x\in\Omega$ and ${\mathcal H}$ is the mean curvature of $\partial\Omega$.
	Then there exists a positive constant $L$ such that
	\begin{equation*}
	|u(x)-g\circ\eta({\rm d}(x,\partial\Omega)-{\mathcal H}(\sigma(x))J({\rm d}(x,\partial\Omega)))|\leq L\T(x)o({\rm d}(x,\partial\Omega)),
	\end{equation*}
	whenever ${\rm d}(x,\partial\Omega)$ goes to zero.
	Furthermore, the following facts hold: 
	\begin{enumerate}
	\item If $p>2k+3$ (even if \eqref{segnomonot} do not hold), then
	$$
	u(x)=\frac{\Gamma}{({\rm d}(x,\partial\Omega))^{\frac{2}{p-k-1}}}(1+o(1)),
	$$
	whenever $x$ approaches $\partial\Omega$.
	\vskip4pt
	\item If \eqref{segnomonot} hold with $R=0$ and $k+1<p\leq 2k+3$, then 
	$$
	u(x)=\Gamma\frac{1}{({\rm d}(x,\partial\Omega))^{\frac{2}{p-k-1}}}(1+o(1))
	+\Gamma'\frac{{\mathcal H}(\sigma(x)) }{({\rm d}(x,\partial\Omega))^{\frac{3+k-p}{p-k-1}}}(1+o(1)),
	$$
	whenever $x$ approaches $\partial\Omega$.
\end{enumerate}
\end{theorem}

\noindent
The proof of Theorem~\ref{blowII} follows by combining Propositions~\ref{precise-potpot} and~\ref{blowII-prop}.

\section{Some remarks}	

\noindent
Some remarks are now in order on the results stated in the previous section.

\begin{remark}[Derivatives blow-up]\rm
According to a result due to Bandle and Marcus \cite[see Section 3, Theorem 3.1]{BM2} 
for general semi-linear equations, not only the solution $u$ blows up 
along the boundary but also the modulus of the gradient $|Du|$ explodes. Hence, concerning problem~\eqref{prob}, a result
in the spirit of Theorem~\ref{blowI} for the gradient could be stated too, under suitable assumptions on the asymptotic 
behavior of $a$ and $f$ yielding to
$$
\lim_{{\rm d}(x,\partial\Omega)\to 0}\frac{a(u)|Du(x)|^2}{2F\circ g\circ\eta({\rm d}(x,\partial\Omega))}= 1,
$$
under suitable assumptions on the domain. In the case \eqref{caso:potpot} this turns into
$$
|Du(x)|=\frac{{\Gamma_\flat}}{({\rm d}(x,\partial\Omega))^{\frac{p+1-k}{p-k-1}}}(1+o(1)),\quad\,\,
\text{for ${\rm d}(x,\partial\Omega)\to 0$},
$$
where $\Gamma_\flat :=\frac {2\Gamma}{p-k-1}$, by exploiting the information provided in \eqref{limetaa}.
\end{remark}	

\begin{remark}[Sign condition on $a$]\rm
Assume that $a$ 
satisfies 
$$
a'(s)s\geq 0,  \,\,\,\quad\text{for every $s\in\R$,}
$$
then a nonnegative solution $u$ of \eqref{prob} satisfies the differential inequality
\begin{equation*}
\,\dvg(a(u)Du)\geq f(u), \,\,\,\,\,\text{in $\Omega$}.
\end{equation*}
In this case the problem become simpler. We will not assume any sign condition on $a$.
\end{remark}

\begin{remark}[Nonnegative solutions]\rm
	\label{signedsol}
Assuming that 
$f(s)=0$ for every $s\leq 0$, any solution to \eqref{prob-semi} is nonnegative (and hence any solution to~\eqref{prob}
is nonnegative). In fact, since $g(0)=0$ and $g'>0$, it is $g(s)\leq 0$ for all $s\leq 0$ and therefore
\begin{equation}
	\label{nullnegg}
h(s)=\frac{f(g(s))}{\sqrt{a(g(s))}}=0,  \,\,\,\quad\text{for every $s\leq 0$}.
\end{equation}
Let $v:\Omega\to\R$ be a classical solution of $\Delta v= h(v)$ 
such that $v(x)\to+\infty$ as $x$ approaches $\partial\Omega$
and define the open bounded set $\Omega_-=\{x\in\Omega: v(x)<0\}$. We have
\begin{equation}
	\label{inclusionebordo}
\partial\Omega_-\subseteq \{x\in\Omega:v(x)=0\}.
\end{equation}
In fact, if $x\in \partial\Omega_-=\overline{\Omega}_-\setminus \Omega_-$, there is a sequence
$(\xi_j)\subset\Omega$ with $\xi_j\to x$ and $v(\xi_j)<0$. It follows that $x\in\Omega$, otherwise
$v(\xi_j)\to+\infty$ if $x\in\partial\Omega$. In addition, $v(x)\leq 0$. Since $x\not\in\Omega_-$, 
we also have $v(x)\geq 0$, proving the claim.
In view of \eqref{nullnegg}, $v$ is harmonic in $\Omega_-$.
Assume by contradiction that $\Omega_-\not =\emptyset$ and let
$\xi\in\Omega_-$. By \eqref{inclusionebordo} and the maximum principle in $\Omega_-$
$$
v(\xi)\geq \min_{\Omega_-}v\geq \min_{\bar\Omega_-}v=\min_{\partial\Omega_-}v=0,
$$
yielding a contradiction. Then $v\geq 0$ on $\Omega$ and for the 
solution $u$ of \eqref{prob}, $u=g(v)\geq 0$. 
\end{remark}

\begin{remark}[Negative large solutions]\rm
In analogy with the study of positively blowing up solutions, it is possible to
formulate existence and nonexistence results for the problem
\begin{equation}
\label{prob-neg}
\begin{cases}
\,\dvg(a(u)Du)=\frac{a'(u)}{2}|Du|^2+f(u)   & \text{in $\Omega$,} \\
\noalign{\vskip2pt}
\,\text{$u(x)\to-\infty$\quad as ${\rm d}(x,\partial\Omega)\to 0$,}    &
\end{cases}
\end{equation}
by assuming that $a$ is an even function, that there exists $r\in\R$ such that $f(r)<0$ and $f(s)\leq 0$ for all $s<r$ and by
prescribing suitable asymptotic conditions on $a$ and $f$ as $s\to-\infty$. Furthermore, by arguing
as in Remark~\ref{signedsol}, it is readily seen that if 
$f(s)=0$ for 
every $s\geq 0$, any solution to \eqref{prob-neg} is nonpositive. See Remark~\ref{rns} in 
Section~\ref{existencesection} for more details on how to detect solutions to~\eqref{prob-neg}
when $a$ is even by reducing the problem to a related one with positive blow-up.
\end{remark}

\begin{remark}[Lower bound of large solutions]\rm
Under assumption \eqref{u1} and \eqref{u2} problem \eqref{prob} has a unique positive solution $u$ in 
$\Omega$ and it is possible to estimate the minimum of 
$u$ in $\Omega$ in terms of the minimum of the unique radial solution $z$ of \eqref{prob} 
in a ball $B$ such that $|B|=|\Omega|$, yielding
\begin{equation}
	\label{r3}
\min_{x\in \Omega}u(x)\geq \min_{x\in B}z(x).
\end{equation}
To prove this, let us consider the associated semi-linear problem \eqref{prob-semi}. 
Quoting a result of \cite{GP} (see Theorem 3.1), we get
$$
\min_{x\in \Omega}v(x)\geq \min_{x\in B} z_\sharp(x)
$$
where $z_\sharp(x)$ is the unique solution of \eqref{prob-semi} in a ball $B$ such that $|B|=|\Omega|$. 
The monotonicity of $g$ then yields \eqref{r3} via Lemma \ref{colleg}.
Moreover, in some cases, the unique radial solution $z_\sharp$ of \eqref{prob-semi} 
in a ball $B$ is explicitly known and this provides an estimate on the minimum 
of $z_\sharp$ (hence of $z$) in terms of $|\Omega|$. 
\end{remark}

\begin{remark}[Convexity of sublevel sets in strictly convex domains]\rm
Assume the domain $\Omega$ is strictly convex and that for a solution $u$ to \eqref{prob} we have: 
\begin{equation}
	\label{ipot-convex}
u>0,\quad  \text{$ 2 f'(s)a(s)-f(s)a'(s)  >0$  for $s>0$}
\quad \text{$s\mapsto \frac{\sqrt{a(g(s))}}{f(g(s))}$\,\,\, is convex on $(0,+\infty)$.}
\end{equation}
Then $g^{-1}(u)$ is strictly convex for $N=2$. The same occurs in higher dimensions provided that 
the Gauss curvature of $\de \Omega$ is strictly positive (see \cite{thorpe}, for example, for the 
definition of Gauss curvature of a surface). In these cases, furthermore, 
the sublevel sets of $u$ are strictly convex. In fact,
we have $u=g(v)$, for some $v\in C^2(\Omega)$, $v>0$, which satisfies $\Delta v=h(v)$, see Lemma~\ref{colleg},
where $h$ is defined as in \eqref{prob-semi}. Furthermore, by \eqref{ipot-convex}, we have that $h>0$, $h$ strictly increasing and  $1/h$  convex.
Let us consider the Concavity function
$$C(v,x,y):=v\big(\frac{x+y}2\big)-\frac 12 v(x)-\frac 12v(y)$$
defined in $\Omega\times \Omega$ 
as introduced in \cite{korevaar} to study the convexity of the level sets of solutions of some semi-linear equations.
We are then in position to apply a result \cite[Theorem 3.13]{kawohl}, by Kawohl, which implies that the Concavity function cannot attain a positive maximum
in $\Omega\times \Omega$. Moreover by \cite[Lemma 3.11]{kawohl}, the Concavity function is negative in a neighborhood of $\de(\Omega\times \Omega)$ so that 
$C(v,x,y)\leq 0$ in $ \Omega\times \Omega$ and hence $v$ is convex in $\Omega$.
If $N=2$, from a result of Caffarelli and Friedman \cite[Corollary 1.3]{caffri}, 
$v$ needs to be strictly convex in $\Omega$. In higher dimensions, assuming that the Gauss 
curvature of $\Omega$ is strictly positive, it is possible to find some points close to
the boundary $\partial\Omega$ where the Hessian matrix of $v$ has full rank. Then, from a result due to Korevaar and Lewis  \cite[Theorem 1]{korv} we would get
that $v$ is strictly convex. In these cases, from the strict monotonicity of $g$, we get that the (closed) sublevel sets of $u$ are strictly convex.
\end{remark}

\begin{remark}[A case of uniqueness]\rm
	\label{remuniexx}
Assume that $p>k+1$ and
\begin{equation}
	\label{expotpot}
f(s)=
\begin{cases}
s^p & \text{if $s\geq 0$},  \\
0 & \text{if $s<0$},
\end{cases}
\,\,\,\qquad
a(s)=1+|s|^k,\,\,\,\,\,\, s\in\R.
\end{equation}
Then condition~\eqref{u1} is satisfied. Moreover, using Lemma \ref{l2.3} with $a_{\infty}=1$, we have
\begin{align*}
\lim_{s\to +\infty}&\frac{\big[ 2f'(s)a(s)-f(s)a'(s)\big]g^{-1}(s)}{2 (a(s))^{\frac 32}f(s)}\\
=&\lim_{s\to +\infty}\frac{2ps^{p-1}(1+s^k)-s^p(ks^{k-1})}{2(1+s^k)^{\frac 32}s^p}\frac{g^{-1}(s)}{s^{\frac {k+2}2}}s^{\frac {k+2}2}\\
=&\lim_{s\to +\infty}\frac{(2p-k)s^{p-k-1}(1+o(1)) s^{\frac {k+2}2}}{2s^{p+\frac 32 k}(1+o(1))}\frac 1{g_{\infty}^{\frac{k+2}2}}\\
=&\frac {(2p-k)}{2}\frac 2{k+2}=\frac {2p-k}{k+2}>1
\end{align*}
since $p>k+1$, so that \eqref{u2} is fulfilled.
In particular 
\begin{equation*}
\begin{cases}
\,\dvg((1+u^k)Du)-\frac{k}{2}u^{k-1}|Du|^2=u^p   & \text{in $\Omega$,} \\
\noalign{\vskip2pt}
\,\text{$u(x)\to+\infty$\quad as ${\rm d}(x,\partial\Omega)\to 0$,}    
\end{cases}
\end{equation*}
admits a unique nonnegative solution in every bounded smooth domain $\Omega$.
\end{remark}

\begin{remark}\rm
Let us observe that, under the existence conditions of Theorem~\eqref{main-ex}, 
condition~\eqref{u1b} is always satisfied. This is proved, for instance, in Section~\ref{se-simm} where we assume that $a$ and $g$ are of class $C^2$. 
Then, if $\Omega$ is of class $C^3$ and has positive mean curvature on $\partial\Omega$ 
the solution of~\eqref{prob} is unique if the function $h$ is nondecreasing.  
\end{remark}

\begin{remark}\rm
Various results appeared in the recent literature about existence and qualitative properties of large solutions for
the $m$-Laplacian equation $\Delta_m u=f(u)$ with $m>1$ on a smooth bounded $\Omega$, see for example \cite{GP}.
On this basis, using a suitable modification of the change of variable
Cauchy problem \eqref{cauchy} and of the Keller-Ossermann condition \eqref{generalKO}, many of the properties stated in our results might be 
extended to cover the study of blow-up solutions of
\begin{equation*}
\,\dvg(a(u)|Du|^{m-2}Du)=\frac{a'(u)}{m}|Du|^m+f(u)   \quad\,\, \text{in $\Omega$.} 
\end{equation*}
\end{remark}

\section{Existence of solutions}
\label{existencesection}

\noindent
Assume that $a:\R\to\R$ is a function of class $C^1$ such that there exists $\nu>0$ with $a(s)\geq \nu$,
	for every $s\in\R$.
The function $g:\R\to\R$, defined in \eqref{cauchy}, is smooth and strictly increasing.
Then, it is possible to associate to problem~\eqref{prob} the semi-linear problem
\begin{equation}
\label{prob-semi}
\begin{cases}
\,\Delta v=h(v)  \quad\text{in $\Omega$,}    &  \\
\noalign{\vskip2pt}
\,\text{$v(x)\to+\infty$\quad as ${\rm d}(x,\partial\Omega)\to 0$,}    &
\end{cases}
\end{equation}
where we let $h(s):=\frac{f(g(s))}{\sqrt{a(g(s))}} $.
\vskip3pt
\noindent
More precisely, we have the following
\begin{lemma}
	\label{colleg}
If $v\in C^2(\Omega)$ is a classical solution to problem~\eqref{prob-semi},
	then $u=g(v)$ is a classical solution to problem~\eqref{prob}. Vice versa, if $u\in C^2(\Omega)$ is a classical
	solution to problem~\eqref{prob}, then $v=g^{-1}(u)$ is a classical solution to problem~\eqref{prob-semi}.
\end{lemma}
\begin{proof}
Observe first that the solution $g$ to the Cauchy problem~\eqref{cauchy} 
is globally defined, of class $C^2$, strictly increasing (thus invertible with inverse $g^{-1}$) and
\begin{equation}
	\label{limiti1}
\lim_{s\to \pm\infty} g(s)=\pm\infty, \,\,\quad
\lim_{s\to \pm\infty} g^{-1}(s)=\pm\infty.
\end{equation}
Now, if $v\in C^2(\Omega)$ is a classical solution to~\eqref{prob-semi}, from the regularity of $g$ it is $u=g(v)\in C^2(\Omega)$.
Furthermore, from~\eqref{limiti1}, it is $u(x)\to+\infty$ as ${\rm d}(x,\partial\Omega)\to 0$. In addition $Dv=a^{1/2}(u)Du$
and $\Delta v=a^{1/2}(u)\Delta u+\frac{1}{2}a'(u)a^{-1/2}(u)|Du|^2$, so that from \eqref{prob-semi} it is
$a(u)\Delta u+\frac{1}{2}a'(u)|Du|^2-f(u)=0$ which easily yields the assertion. Vice versa, if $u$ is a classical solution to~\eqref{prob},
then $v=g^{-1}(u)$ is of class $C^2$ and by \eqref{limiti1}, it is $v(x)\to+\infty$ as ${\rm d}(x,\partial\Omega)\to 0$. Moreover,
it follows that $\Delta v=a^{1/2}(u)\Delta u+\frac{1}{2}a'(u)a^{-1/2}(u)|Du|^2=f(u)a^{-1/2}(u)=f(g(v))a^{-1/2}(g(v))$, concluding the proof.
\end{proof}

\noindent
On the basis of Lemma~\ref{colleg}, we try to establish existence of solutions of problem~\eqref{prob} using existence
condition known in the literature for semi-linear problems like \eqref{prob-semi}. 
\vskip5pt
\noindent
Hereafter we let  $f:\R\to\R$ be a $C^1$ function. Consider the following condition:
\vskip2pt
\noindent
{\bf E.} There exists $r\in\R$ such that $f(r)>0$ and $f(s)\geq 0$ for all $s>r$ and
\begin{equation}
\label{generalKO}
\int_{g^{-1}(r)}^{+\infty}\frac{1}{\sqrt{F\circ g}}<+\infty,\qquad\,\,
F(t):=\int_{r}^t f(\tau)d\tau.
\end{equation}
Essentially, {\bf E} depends upon the asymptotic behavior
of the function $F$ and $g$. 

\begin{proposition}
	\label{exprop-gen}
Let $\Omega$ be any smooth bounded domain in $\R^N$. Then problem~\eqref{prob} admits a solution
if and only if {\bf E} holds. 
\end{proposition}

\noindent
Proposition~\ref{exprop-gen} readily follows by combining Lemma~\ref{colleg} with the assertion of \cite[Theorem 1.3]{ddgr}, 
where the authors proved the equivalence between the Keller-Osserman
condition, the sharpened Keller-Osserman condition and the existence of blow-up solutions in arbitrary bounded domains
without requiring any monotonicity assumption on the nonlinearity.
\vskip4pt
\noindent
We shall now investigate the asymptotic behavior of $g$ as 
$s\to+\infty$ according to the cases when $a$ behaves like a polynomial, an
exponential function or a logarithmic function. In turn, in these situation, 
we discuss the validity of condition~\eqref{generalKO}.

\subsection{Polynomial growth}
Assume that there exists $a_\infty>0$ such that
\begin{equation}
	\label{1.7}
\lim_{s\to +\infty}\frac{a(s)} {s^k}=a_\infty.
\end{equation}
\noindent
In \cite[Lemma 2.1]{glasqu1}, we proved the following

\begin{lemma}\label{l2.3}
Assume that condition \eqref{1.7} holds. Then, we have
\begin{equation}\label{1.8}
\lim_{s\to +\infty}\frac{g(s)}{s^{\frac 2{k+2}}}=g_{\infty}, \qquad
\lim_{t\to +\infty}\frac {g^{-1}(t)}{t^{\frac {k+2}2}} =g_{\infty}^{-\frac {k+2}2},
\end{equation}
where
$g_{\infty}= \big(\frac{k+2} 2 \frac 1{\sqrt a_\infty} \big)^{\frac 2{k+2}}$.
\end{lemma}

\noindent
We can now formulate the following existence result.

\begin{proposition}[$f$ with exponential growth]
	\label{potexp}
Assume that~\eqref{1.7} holds and that there exist $\beta>0$ and $f_{\infty} >0$ such that
\begin{equation}
	\label{fassexppot}
\lim_{s\to +\infty}\frac{f(s)} {e^{2\beta s}}=f_{\infty}.
\end{equation}
Then~\eqref{prob} always admits a solution in any smooth domain $\Omega$.
\end{proposition}
\begin{proof}
	In light of assumption~\eqref{fassexppot}, there exists $r>0$ such that $f(t)>0$ for all $t\geq r$. 
        Since
	$$
	\lim_{s\to+\infty}\frac{F(s)}{e^{2\beta s}}=\frac{f_{\infty}}{2\beta},
	$$
	on account of Lemma~\ref{l2.3}, for any $\eps>0$ we obtain
	\begin{align*}
	\lim_{s\to+\infty}\frac{ F(g(s))}{e^{(2\beta g_{\infty}-2\eps) s^{\frac{2}{k+2}}}}&=	
	\lim_{s\to+\infty}\frac{f_{\infty} e^{2\beta g_{\infty}s^{\frac{2}{2+k}}(1+o(1))}(1+o(1))}{2\beta e^{(2\beta g_{\infty} -2\eps)s^{\frac{2}{k+2}}}} \\
	&=\frac {f_{\infty}}{2\beta}\lim_{s\to+\infty} e^{2 s^{\frac{2}{2+k}}(\eps+o(1))}(1+o(1))=+\infty.	
	\end{align*}
	In particular, having fixed $\eps$ such that $\eps<\beta g_{\infty}$, there exists $R=R(\eps)>0$ such that
	$$
	\sqrt{F(g(s))}\geq e^{(\beta g_{\infty}-\eps) s^{\frac{2}{k+2}}},  \,\,\,\quad\text{for every $s\geq R$.}
	$$
Therefore,
$$
\int_{g^{-1}(r)}^{+\infty}\frac{1}{\sqrt{F(g(s))}}ds\leq 
\int_{g^{-1}(r)}^{R}\frac{1}{\sqrt{F(g(s))}}ds
+  \int_{R}^{+\infty}\frac{1}{e^{(\beta g_{\infty}-\eps) s^{2/(k+2)}}}ds<+\infty.
$$
The assertion then follows by Proposition~\ref{exprop-gen}.
\end{proof}

\noindent
Next, we have the following

\begin{proposition}[$f$ with polynomial growth]
	\label{potpot}
Assume that~\eqref{1.7} holds and that there exist $p>1$ and $f_{\infty}>0$ such that
\begin{equation}
	\label{fasspotpot}
\lim_{s\to +\infty}\frac{f(s)} {s^p}= f_{\infty}.
\end{equation}
Then~\eqref{prob} has admits a solution in any smooth domain $\Omega$ if and only if $p>k+1$.
\end{proposition}
\begin{proof}
	In light of \eqref{fasspotpot}, there exists $r>0$ such that $f(t)>0$ for all $t\geq r$. 
        Since
	$$
	\lim_{s\to+\infty}\frac{F(s)}{s^{p+1}}=\frac{ f_{\infty}}{p+1},
	$$
	on account of Lemma~\ref{l2.3},	we get
	\begin{align*}
	& \lim_{s\to+\infty}\frac{F(g(s))}{s^{\frac{2(p+1)}{k+2}}}=
	\lim_{s\to+\infty}\frac{ f_{\infty}(g(s))^{p+1}(1+o(1))}{(p+1)s^{\frac{2(p+1)}{k+2}}}  \\
	& =\frac{f_{\infty}}{p+1}\lim_{s\to+\infty}\frac{(g_{\infty}s^{\frac{2}{k+2}}(1+o(1)))^{p+1}(1+o(1))}{s^{\frac{2(p+1)}{k+2}}}=\frac{f_{\infty}g_{\infty}^{p+1}}{p+1}
	\end{align*}
	In particular, there exists $R>0$ such that
	$$
	\frac{\sqrt{ f_{\infty}}g_{\infty}^{\frac{p+1}{2}}}{\sqrt{2p+2}} s^{\frac{p+1}{k+2}}\leq 
	\sqrt{F(g(s))}\leq \frac{\sqrt{2 f_{\infty}}g_{\infty}^{\frac{p+1}{2}}}{\sqrt{p+1}} s^{\frac{p+1}{k+2}}, \,\,\,\quad\text{for every $s\geq R$.}
	$$
	Then, if $p>k+1$, we have
	$$
	\int_{g^{-1}(r)}^{+\infty}\frac{1}{\sqrt{F(g(s))}}ds\leq 
	\int_{g^{-1}(r)}^{R}\frac{1}{\sqrt{F(g(s))}}ds
	+ \frac{\sqrt{2p+2}}{\sqrt{ f_{\infty}}g_{\infty}^{\frac{p+1}{2}}} 
	\int_{R}^{+\infty}
        s^{-\frac{p+1}{k+2}}ds<+\infty,
	$$
	so that {\bf E} holds true. On the contrary, assuming that $p\leq k+1$, 
	for every $r\in\R$ with $f(r)>0$ and $f(t)\geq 0$ for all $t>r$, we obtain 
	$$
	\int_{g^{-1}(r)}^{+\infty}\frac{1}{\sqrt{F(g(s))}}ds\geq 
	 \frac{\sqrt{p+1}}{\sqrt{2 f_{\infty}}g_{\infty}^{\frac{p+1}{2}}} 
	\int_{\max\{R,g^{-1}(r)\}}^{+\infty}
        s^{-\frac{p+1}{k+2}}ds=+\infty.
	$$
	The assertion then follows by Proposition~\ref{exprop-gen}.
\end{proof}

\noindent
We also have the following

\begin{proposition}[$f$ with logarithmic growth]
	\label{pollog}
Assume that~\eqref{1.7} holds and that there exist $\beta>0$ and $f_{\infty}>0$ such that
\begin{equation}
	\label{log}
\lim_{s\to +\infty}\frac{f(s)} {(\log s)^\beta}= f_{\infty}.
\end{equation}
Then, in any smooth domain $\Omega$, problem~\eqref{prob} admits no solution.
\end{proposition}
\begin{proof}
	The proof proceeds just like in the proof of Proposition~\ref{potpot} observing
	that for any $k>0$ there exists $p_0<k+1$ such that $\sqrt{F(g(s))}\leq  s^{\frac{p_0+1}{k+2}}$ for
	$s$ large.
\end{proof}
\vskip4pt
\noindent
\subsection{Exponential growth}
Assume now that there exist $\gamma>0$ and  $a_{\infty}>0$ such that
\begin{equation}
	\label{exgrow}
\lim_{s\to +\infty}\frac{a(s)} {e^{2\gamma s}}=a_{\infty}.
\end{equation}

\noindent
Then we have the following

\begin{lemma}\label{growl}
Assume that condition \eqref{exgrow} holds. Then, we have
\begin{equation}\label{exlimm}
\lim_{s\to +\infty}\frac{g(s)}{\log s}=\frac{1}{\gamma},\qquad
\lim_{t\to +\infty}\frac {g^{-1}(t)}{e^{\gamma t}} =\frac{\sqrt{a_{\infty}}}{\gamma}.
\end{equation}
\end{lemma}
\begin{proof}
From the definition of $g$, we have $g'(s)a^{1/2}(g(s))=1$. Hence, integrating on $[0,s]$ yields
$$
s=\int_0^s g'(\sigma)a^{1/2}(g(\sigma))d\sigma=\int_0^{g(s)} a^{1/2}(\sigma)d\sigma.
$$
In turn, we reach
\begin{align*}
 \lim_{s\to+\infty}\frac{g(s)}{\frac{1}{\gamma} \log s}&=\gamma\lim_{s\to+\infty} sg'(s)=
\gamma\lim_{s\to+\infty} \frac{s}{\sqrt{a(g(s))}} \\
& =\gamma\lim_{s\to+\infty} \frac{s}{\sqrt{a_{\infty}e^{2\gamma g(s)}(1+o(1))}}=
\frac{\gamma}{\sqrt{a_{\infty}}}\lim_{s\to+\infty} \frac{\int_0^{g(s)} a^{1/2}(\sigma)d\sigma}{e^{\gamma g(s)}} \\
\noalign{\vskip2pt}
&=\frac{\gamma}{\sqrt{a_{\infty}}}\lim_{s\to+\infty} \frac{a^{1/2}(g(s))g'(s)}{\gamma e^{\gamma g(s)}g'(s)} 
=\frac{1}{\sqrt{a_{\infty}}}\lim_{s\to+\infty} \frac{a^{1/2}(g(s))}{e^{\gamma g(s)}}=1, 
\end{align*}
in light of condition \eqref{exgrow}. Furthermore, taking into account the above computations, 
\begin{equation*}
	\lim_{t\to+\infty}\frac{g^{-1}(t)}{e^{\gamma t}}=\lim_{s\to+\infty}\frac{s}{e^{\gamma g(s)}}
	=\frac{\sqrt{a_{\infty}}}{\gamma},
\end{equation*}
concluding the proof.
\end{proof}

\noindent
We can now formulate the following existence result.

\begin{proposition}[$f$ with exponential growth]
	\label{expexp}
Assume that~\eqref{exgrow} holds and that there exist $\beta>0$ and $ f_{\infty}>0$ such that
\begin{equation}
	\label{fassexp}
\lim_{s\to +\infty}\frac{f(s)} {e^{2\beta s}}=f_{\infty}.
\end{equation}
Then~\eqref{prob} admits a solution in any smooth domain $\Omega$ if and only if $\beta>\gamma$.
\end{proposition}
\begin{proof}
In light of \eqref{fassexp}, there exists $r>0$ such 
that $f(t)>0$ for all $t\geq r$. Furthermore, 
taking into account Lemma~\ref{growl} and that
$$
\lim_{s\to+\infty}\frac{F(s)}{e^{2\beta s}}=\frac{f_{\infty}}{2\beta},
$$
we have in turn, for any $\eps>0$, 
\begin{align*}
\lim_{s\to+\infty}\frac{ F(g(s))}{s^{\frac{2\beta}{\gamma}-2\eps}}& 
=\lim_{s\to+\infty} \frac {f_{\infty}}{2\beta}\frac{  e^{2\beta g(s)}(1+o(1))}{s^{\frac{2\beta}{\gamma}-2\eps}}  \\
&=\frac {f_{\infty}}{2\beta} \lim_{s\to+\infty}\frac{ e^{\frac{2\beta}{\gamma}(1+o(1))\log s}(1+o(1))}{s^{\frac{2\beta}{\gamma}-2\eps}} \\
&= \frac {f_{\infty}}{2\beta} \lim_{s\to+\infty}\frac{ s^{\frac{2\beta}{\gamma}}s^{o(1)}}{s^{\frac{2\beta}{\gamma}-2\eps}}=\frac {f_{\infty}}{2\beta} 
\lim_{s\to+\infty}s^{2\eps+o(1)}=+\infty.
\end{align*}
Consider the case $\beta>\gamma$. We can fix $\eps$ in such a way that $0<\eps<\frac{\beta}{\gamma}-1$. 
Corresponding to this choice of $\eps$ there exists $R=R(\eps)>0$ large enough that
$$
F(g(s)) \geq s^{\frac{2\beta}{\gamma}-2\eps},  \,\,\,\quad\text{for every $s\geq R$,}
$$
yielding in turn, 
$$
\int_{g^{-1}(r)}^{+\infty}\frac{1}{\sqrt{F(g(s))}}ds\leq 
\int_{g^{-1}(r)}^{R}\frac{1}{\sqrt{F(g(s))}}ds+ \int_{R}^{+\infty}\frac{1}{s^{\frac{\beta}{\gamma}-\eps}}ds<+\infty,
$$
so that {\bf E } is fulfilled. Assume now that $\beta<\gamma$. Given $\eps>0$, by arguing as above, we obtain 
\begin{equation*}
\lim_{s\to+\infty}\frac{ F(g(s))}{s^{\frac{2\beta}{\gamma}+2\eps}}
=\lim_{s\to+\infty} \frac {f_{\infty}}{2\beta}\frac{ s^{\frac{2\beta}{\gamma}}s^{o(1)}}{s^{\frac{2\beta}{\gamma}+2\eps}}= 
\frac {f_{\infty}}{2\beta} \lim_{s\to+\infty}\frac{s^{o(1)}}{s^{2\eps}}=0.
\end{equation*}
Hence, fixed $\bar\eps>0$ sufficiently small that $\frac {\beta}{\gamma}+\bar\eps<1$, 
there exists $R=R(\bar\eps)> 0$ such that
$$
F(g(s)) \leq s^{\frac{2\beta}{\gamma}+2\bar\eps},  \,\,\,\quad\text{for every $s\geq R$.}
$$
Thus, for every $r\in\R$ with $f(r)>0$ and $f(t)\geq 0$ for all $t>r$, we obtain
$$
\int_{g^{-1}(r)}^{+\infty}\frac{1}{\sqrt{F(g(s))}}ds\geq 
\int_{\max\{R,g^{-1}(r)\}}^{+\infty}\frac{1}{s^{\frac{\beta}{\gamma}+\bar\eps}}ds=+\infty.
$$
The assertion then follows by Proposition~\ref{exprop-gen}.
\end{proof}

\noindent
Next, we formulate the following existence result.

\begin{proposition}[$f$ with polynomial growth]
	\label{exppot}
Assume that~\eqref{exgrow} holds and that there exist $p>1$ and $ f_{\infty}>0$ such that
\begin{equation}
	\label{faspot}
\lim_{s\to +\infty}\frac{f(s)} {s^p}=f_{\infty}.
\end{equation}
Then, in any smooth domain $\Omega$, problem~\eqref{prob} admits no solution.
\end{proposition}
\begin{proof}
        Since
	$$
	\lim_{s\to+\infty}\frac{F(s)}{s^{p+1}}=\frac{f_{\infty}}{p+1},
	$$
	there exist $\beta<\gamma$ and $R>0$ such that
	$$
	F(s)\leq 
        e^{2\beta s}, \,\,\,\quad\text{for every $s\geq R$}.
	$$
	Fixed now $\bar\eps>0$ so small that $\frac{\beta}{\gamma}(1+\bar\eps)<1$, taking into account Lemma~\ref{growl}, 
there exists $R=R(\bar\eps)>0$ 
such that 
	$$
	\sqrt{F(g(s))}\leq 
        e^{{\beta}g(s)}
	=
        e^{\beta \frac{g(s)}{\log s}\log s}
        \leq 
        s^{\frac{\beta}{\gamma}(1+\bar\eps)},
	$$
	for all $s\geq R$. Then, for every $r\in\R$ with $f(r)>0$ and $f(t)\geq 0$ for all $t>r$, we have 
	$$
	\int_{g^{-1}(r)}^{+\infty}\frac{1}{\sqrt{F(g(s))}}ds\geq 
        \int_{\max\{R,g^{-1}(r)\}}^{+\infty}\frac{1}{s^{\frac{\beta}{\gamma}(1+\bar\eps)}}ds=+\infty.
	$$
	The assertion then follows by Proposition~\ref{exprop-gen}.
\end{proof}

\noindent
We also have the following

\begin{proposition}[$f$ with logarithmic growth]
	\label{explog}
Assume that~\eqref{exgrow} holds and that there exist $\beta>0$ and $ f_{\infty}>0$ such that
\begin{equation}
	\label{faslog}
\lim_{s\to +\infty}\frac{f(s)} {(\log s)^\beta}=f_{\infty}.
\end{equation}
Then, in any smooth domain $\Omega$, problem~\eqref{prob} admits no solution.
\end{proposition}
\begin{proof}
	The proof proceeds as for Proposition~\ref{exppot} since
	$F(s)\leq e^{ 2\beta s}$ with $\beta<\gamma$ for $s$ large.
\end{proof}

\subsection{Logarithmic growth}

Assume now that there exist $\gamma>0$ and  $a_{\infty}>0$ such that
\begin{equation}
	\label{loggrow}
\lim_{s\to +\infty}\frac{a(s)} {(\log s)^{2\gamma}}=a_{\infty}.
\end{equation}

Then we have the following

\begin{lemma}\label{growlog}
Assume that condition \eqref{loggrow} holds. Then, we have
\begin{equation}\label{exlimmlog}
\lim_{s\to +\infty}\frac{g(s)}{s^{1-\eps}}=+\infty,\qquad
\lim_{s\to +\infty}\frac{g(s)}{s}=0.
\end{equation}
for every $\eps\in (0,1)$. In particular, for every $\eps\in(0,1)$ there exists $R=R(\eps)>0$ such that
\begin{equation}
	\label{logineq}
	s^{1-\eps}\leq g(s)\leq s,\quad\text{for every $s\geq R$.}
\end{equation}
\end{lemma}
\begin{proof}
	We have, for every $\eps\in (0,1)$
	\begin{align*}
		\lim_{s\to+\infty}\frac{g(s)}{s^{\frac{1}{1+\eps}}}&=\lim_{s\to+\infty}\Big(\frac{g^{1+\eps}(s)}{s}\Big)^{\frac{1}{1+\eps}} 
=\Big(\lim_{s\to+\infty}\frac{g^{1+\eps}(s)}{\int_0^{g(s)} \sqrt{a(\sigma)}d\sigma}\Big)^{\frac{1}{1+\eps}}	\\
&=\Big(\lim_{s\to+\infty}\frac{(1+\eps)g^{\eps}(s)g'(s)}{\sqrt{a(g(s))}g'(s)}\Big)^{\frac{1}{1+\eps}} \\
&=
(1+\eps)^{\frac{1}{1+\eps}}\Big(\lim_{s\to+\infty}\frac{g^{\eps}(s)}{(\log g(s))^{\gamma}}\frac{(\log g(s))^{\gamma}}{\sqrt{a(g(s))}}\Big)^{\frac{1}{1+\eps}}=+\infty.
	\end{align*}
	Furthermore, we have
	\begin{align*}
&		\lim_{s\to+\infty}\frac{g(s)}{s}=\lim_{s\to+\infty}\frac{g(s)}{\int_0^{g(s)} a^{1/2}(\sigma)d\sigma}=
\lim_{s\to+\infty}\frac{1}{a^{1/2}(g(s))}=0,
	\end{align*}	
	which concludes the proof of the lemma.
\end{proof}

\noindent
We can now formulate the following existence result.

\begin{proposition}[$f$ with exponential growth]
	\label{logexp}
Assume that~\eqref{loggrow} holds and that there exist $\beta>0$ and $ f_{\infty}>0$ such that
\begin{equation}
	\label{fassexplog}
\lim_{s\to +\infty}\frac{f(s)} {e^{2\beta s}}=f_{\infty}.
\end{equation}
Then~\eqref{prob} admits a solution in any smooth domain $\Omega$.
\end{proposition}
\begin{proof}
	By~\eqref{fassexplog}, there exists $r>0$ such 
	that $f(t)>0$ for all $t\geq r$.  
	By virtue of Lemma~\ref{growlog}, for every $\eps\in (0,1)$ there exists $R=R(\eps)>0$ such that	
	$$
	F(g(s))=\frac{ f_{\infty}}{2\beta}e^{2\beta g(s)}(1+o(1))\geq \frac{f_{\infty}}{4\beta}e^{2\beta g(s)}\geq \frac{f_{\infty}}{4\beta}e^{2\beta s^{1-\eps}},
	$$
	for every $s\geq R$. In turn, for any $\eps\in (0,1)$, we conclude
	$$
	\int_{g^{-1}(r)}^{+\infty}\frac{1}{\sqrt{F(g(s))}}ds\leq 
	\int_{g^{-1}(r)}^{R}\frac{1}{\sqrt{F(g(s))}}ds+\frac{2\sqrt{\beta}}{\sqrt{f_{\infty}}}\int_{R}^{+\infty}\frac{1}{e^{\beta s^{1-\eps}}}ds<+\infty,
	$$
	concluding the proof of {\bf E}. The assertion then follows by Propositions~\ref{exprop-gen}.
\end{proof}

\noindent
Next we state the following existence result.

\begin{proposition}[$f$ with polynomial growth]
	\label{logpot}
Assume that~\eqref{loggrow} holds and that there exist $p>0$ and $f_{\infty}>0$ such that
\begin{equation}
	\label{fasspotlog}
\lim_{s\to +\infty}\frac{f(s)} {s^p}=f_{\infty}.
\end{equation}
Then~\eqref{prob} admits a solution in any smooth domain $\Omega$ if and only if $p>1$.
\end{proposition}
\begin{proof}
	By~\eqref{fasspotlog}, there exists $r>0$ such 
	that $f(t)>0$ for all $t\geq r$.
	By virtue of Lemma~\ref{growlog}, for every $\eps\in (0,1)$ there exists $R=R(\eps)>0$ such that	
	$$
	F(g(s))=\frac{f_{\infty}}{p+1}g^{p+1}(s)(1+o(1))\geq \frac{f_{\infty}}{2p+2}g^{p+1}(s)\geq \frac{f_{\infty}}{2p+2}s^{(p+1)(1-\eps)},
	$$
	for every $s\geq R$. In turn, if $p>1$, fixed $\eps\in (0,1)$ with $\frac{(p+1)(1-\eps)}{2}>1$ , we conclude
	$$
	\int_{g^{-1}(r)}^{+\infty}\frac{1}{\sqrt{F(g(s))}}ds\leq 
	\int_{g^{-1}(r)}^{R}\frac{1}{\sqrt{F(g(s))}}ds+\frac{\sqrt{2p+2}}{\sqrt{f_{\infty}}}\int_{R}^{+\infty}s^{-\frac{(p+1)(1-\eps)}{2}}ds<+\infty.
	$$
	If $p\leq 1$, exploiting again Lemma~\ref{growlog}, we can find $R$ such that 
	$$
	F(g(s))=\frac{f_{\infty}}{p+1}g^{p+1}(s)(1+o(1))\leq \frac{2f_{\infty}}{p+1}s^{p+1}
	 \,\,\,\quad\text{for all $s\geq R$}
	$$
	yielding, as $\frac{p+1}{2}\leq 1$, for every $r\in\R$ with $f(r)>0$ and $f(t)\geq 0$ for $t>r$,
	$$
	\int_{g^{-1}(r)}^{+\infty}\frac{1}{\sqrt{F(g(s))}}ds\geq 
	\frac{\sqrt{p+1}}{\sqrt{2f_{\infty}}}\int_{\max\{R,g^{-1}(r)\}}^{+\infty}{s^{-\frac{p+1}{2}}}ds=+\infty,
	$$
	concluding the proof of condition {\bf E}. The assertion then follows by Proposition~\ref{exprop-gen}. 
\end{proof}

\noindent
Finally, we have the following

\begin{proposition}[$f$ with logarithmic growth]
	\label{loglog}
Assume that~\eqref{loggrow} holds and that there exist $\beta>0$ and $f_{\infty}>0$ such that
\begin{equation}
	\label{logloggr}
\lim_{s\to +\infty}\frac{f(s)} {(\log s)^\beta}=f_{\infty}.
\end{equation}
Then, in any smooth domain $\Omega$, problem~\eqref{prob} admits no solution.
\end{proposition}
\begin{proof}
	Taking into account Lemma~\ref{growlog}, the proof proceeds as for Proposition~\ref{logpot} since there exists $p_0<1$ such that
	$F(s)\leq s^{p_0+1}$ for every $s$ large.
\end{proof}
\vskip4pt
\begin{remark}[Negative large solutions II] \label{rns}\rm
Assume that $a$ is even 
and consider the following condition:
	\vskip2pt
	\noindent
	{\bf E-.} There exists $r\in\R$ such that $f(r)<0$ and $f(s)\leq 0$ for all $s<r$ and
	\begin{equation}
	\label{generalKOneg}
	\int_{-\infty}^{g^{-1}(r)}\frac{1}{\sqrt{F\circ g}}<+\infty,\qquad\,\,
	F(t):=\int_{r}^t f(\tau)d\tau.
	\end{equation}	
	Then problem \eqref{prob-neg} has a solution
	if and only if {\bf E-} holds. In fact, being $a$ even, it is readily seen that $g$ is odd, and letting  
	$$
	f_0(s):=-f(-s)\quad\text{for all $s\in\R$},\qquad
	F_0(t)=\int_{-r}^t f_0(\tau)d\tau, 
	$$
        two changes of variable yield
	$$
	\int_{g^{-1}(-r)}^{+\infty}\frac{1}{\sqrt{F_0\circ g}}=\int_{-\infty}^{g^{-1}(r)}\frac{1}{\sqrt{F\circ g}}.
	$$	
	Therefore {\bf E-} holds for problem~\eqref{prob-neg} 
	if and only if {\bf E} holds for the problem
	\begin{equation}
		\label{prob-neg-f0}
	\begin{cases}
	\,\dvg(a(w)Dw)-\frac{a'(w)}{2}|Dw|^2=f_0(w)   & \text{in $\Omega$,} \\
	\noalign{\vskip2pt}
	\,\text{$w(x)\to+\infty$\quad as ${\rm d}(x,\partial\Omega)\to 0$,}    &
	\end{cases}
	\end{equation}
	in which case \eqref{prob-neg-f0} admits a solution. Then $u:=-w$ is a solution to~\eqref{prob-neg}.
\end{remark}

\section{Uniqueness of solutions}

\noindent
Concerning the uniqueness for the solutions of problems~\eqref{prob} and~\eqref{prob-ball}, we have the following

\begin{proposition}
	\label{uniglob}
If~\eqref{u1} and~\eqref{u2} are satisfied then problem~\eqref{prob} admits a unique solution which is nonnegative. 
If, else, the conditions which guarantee the 
existence of solutions are fulfilled, if $\partial \Omega$ is
of class $C^3$ and its mean curvature is nonnegative
and~\eqref{u1} and~\eqref{u1b} are satisfied, then problem~\eqref{prob} admits a unique solution which is nonnegative. 
Finally, if the conditions which guarantee the 
existence of solutions are fulfilled and~\eqref{uniqcond} is 
satisfied then problem~\eqref{prob-ball} admits a unique solution. 
\end{proposition}

\begin{proof}
According to \cite[Theorem 1]{GM} problem~\eqref{prob-semi} has a unique 
non-negative solution in a smooth bounded domain $\Omega$ if $h(0)=0$ 
and $h'(s)\geq 0$ for any $s\geq 0$ and if there exist $m>1$ and $t_0>0$ such that 
$\frac {h(t)}{t^m}$ is increasing for $t\geq t_0$. The second 
hypothesis (which guarantees the existence of the solution) is equivalent to require that $\lim_{t\to +\infty}\frac{t h'(t)}{h(t)}> 1$. 
Recalling that 
$h(s)=f(g(s))a^{-1/2}(g(s))$, we have
\begin{equation}
	\label{hprimo}
h'(s)=\frac{2f'(g(s))a(g(s))-f(g(s))a'(g(s))}{2a^2(g(s))},\qquad\text{for every $s\in\R$}.
\end{equation}
In turn, the uniqueness conditions of \cite{GM} turn 
into~\eqref{u1} and~\eqref{u2} which readily yields the desired conclusion since $g$ is a bijection.\\
\noindent 
Now we quote a result of Costin, Dupaigne and Goubet \cite[Theorem 1.3]{CDG}, which says that, 
under smoothness assumption on $\partial \Omega$ and positivity of its mean curvature, if $h$ is nondecreasing and $\sqrt h$ is 
asymptotically convex then the solution of problem~\eqref{prob-semi} is unique. 
Conditions~\eqref{u1} and~\eqref{u1b} allow us to use their result and provide the 
uniqueness of the solution $u$ of problem~\eqref{prob} via the transformation~\eqref{cauchy}.\\
In the case of $\Omega=B_1(0)$, according to \cite[Corollary 1.4]{CD},
the uniqueness of large solutions of $\Delta v=h(v)$ in $B_1(0)$
is guaranteed provided that the existence conditions are satisfied and 
the map $\{s\mapsto h(s)+\lambda_1 s\}$ is nondecreasing on $\R$. 
In turn, the uniqueness condition turns into
$$
2f'(g(s))a(g(s))-f(g(s))a'(g(s))+2\lambda_1 a^2(g(s))\geq 0,\qquad\text{for every $s\in\R$}.
$$
which readily yields the desired conclusion since $g$ is a bijection.
\end{proof}

\section{Nonuniqueness of solutions}

In this section we discuss the existence of two distinct solutions to \eqref{prob} under suitable 
assumptions on $a$ and $f$.
By virtue of \cite[Theorem 1]{AR}, we have the following

\begin{proposition} 
\label{nonuniq}
Let $\Omega$ be bounded, convex and $C^2$. Assume that condition {\bf E} holds, $f(0)=0$ and, for some $R>0$ large,
\begin{equation}
	\label{multiass}
f|_{(R,+\infty)}>0,\qquad
\Big(\frac{f'}{f}-\frac{a'}{2a}\Big)|_{(R,+\infty)}\geq 0,
\end{equation}
and there exists $1<q<\frac{N+2}{N-2}$ if $N\geq 3$ and $q>1$ for $N=1,2$ such that
\begin{equation}
	\label{multiass2}
0<\lim_{s\to-\infty}\frac{f(s)}{\sqrt{a(s)}|g^{-1}(s)|^q}<+\infty.
\end{equation}
Then problem~\eqref{prob} admits at least two distinct solutions, one positive and one sign-changing.
\end{proposition}
\begin{proof}
The function $h(s)=f(g(s))\sqrt{a^{-1/2}(g(s))}$ is smooth and $h(0)=0$. Moreover, 
recalling that $g(s)\to+\infty$ as $s\to+\infty$, by virtue of \eqref{multiass}
there exists $R>0$ sufficiently large that
$$
h(s)>0,\qquad 
\frac{f'(g(s))}{f(g(s))}\geq \frac{a'(g(s))}{2a(g(s))}, \,\,\quad\text{for any $s\geq R$,}
$$
namely $h|_{(R,+\infty)}>0$ and $h'|_{(R,+\infty)}\geq 0$. Finally, due to~\eqref{multiass2} and~\eqref{limiti1}, we get
$$
\lim_{s\to-\infty}\frac{h(s)}{|s|^q}=\lim_{s\to-\infty}\frac{f(g(s))}{\sqrt{a(g(s))}|s|^q}=
\lim_{s\to-\infty}\frac{f(s)}{\sqrt{a(s)}|g^{-1}(s)|^q}\in (0,+\infty).
$$
Whence, in light of \cite[Theorem 1]{AR} we find two distinct large solutions $v_1>0$ and $v_2$ (with $v_2^-\neq 0$ and $v_2^+\neq 0$) 
to the problem $\Delta v=h(v)$. In turn, via Lemma~\ref{colleg}, $u_1=g(v_1)>0$ and $u_2=g(v_2)$ 
(with $u_2^\pm=(g(v_2))^\pm=g(v_2^\pm)\neq 0$) are two distinct explosive solutions of
problem~\eqref{prob}. This concludes the proof.
\end{proof}

\begin{proposition}
	\label{nonuniq-special}
Consider assumptions~\eqref{specialnonu1}-\eqref{specialnonu4} in Proposition~\ref{nonuniq}.
Then problem~\eqref{prob}
admits two solutions, one positive and one sign-changing in any $C^2$ convex and bounded domain.
\end{proposition}
\begin{proof}
It suffices to verify that under condition \eqref{specialnonu1}-\eqref{specialnonu4}, 
assumptions~\eqref{multiass}-\eqref{multiass2} of Proposition~\ref{nonuniq} are fulfilled. Let us observe first, that assumptions~\eqref{multiass} and \eqref{multiass2} imply \eqref{caso:potpot}, so that condition {\bf E} holds. In light of \eqref{specialnonu4} it is readily seen that
the left condition in \eqref{multiass} is satisfied, for some $R>0$ large enough. Moreover, by combining \eqref{specialnonu1} and \eqref{specialnonu4}
and recalling that $p_+>k+1>k/2$ it follows that also the right condition in \eqref{multiass} is satisfied, up to enlarging $R$.
Concerning~\eqref{multiass2}, recalling Lemma~\ref{l2.3}, \eqref{specialnonu2}-\eqref{specialnonu3} and the fact that $g^{-1}$ is odd, 
choosing
$$
1<q:=\frac{2p_{-}-k}{k+2}<\frac{N+2}{N-2}, \,\,\, N\geq 3,
\quad q>1, \,\,\, N=1,2,
$$
we have
$$
\lim_{s\to-\infty}\frac{f(s)}{\sqrt{a(s)}|g^{-1}(s)|^q}=
\lim_{s\to-\infty}\frac{f(s)}{|s|^{p_-}}\lim_{s\to-\infty}\frac{|s|^{k/2}}{\sqrt{a(s)}}
\lim_{s\to-\infty}\frac{|s|^{\frac{q(k+2)}{2}}}{|g^{-1}(s)|^q}
\in(0,+\infty),
$$
concluding the proof.
\end{proof}

\section{Symmetry of solutions} \label{se-simm}

\noindent
Concerning a first condition for the symmetry for the solutions of problem~\eqref{prob} in the ball $B_1(0)$, we have the following

\begin{proposition}
	\label{symm-global}
Assume that 
the conditions which guarantee the 
existence of solutions are fulfilled 
and that there exists $\rho\in\R$ such that
\begin{equation}
	\label{simcond}
	2f'(s)a(s)-f(s)a'(s)+2\rho a^2(s)\geq 0,\qquad \text{for all $s\in\R$.}
\end{equation}
Then any solution to problem~\eqref{prob-ball} 
is radially symmetric and increasing.
\end{proposition}

\begin{proof}
According to \cite[Corollary 1.7]{CD} the symmetry of large solutions of $\Delta v=h(v)$ in $B_1(0)$
is guaranteed provided that the existence conditions are satisfied 
the map $\{s\mapsto h(s)+\rho s\}$ is nondecreasing on $\R$ for some $\rho\in\R$.
Then, the assertion follows arguing as in the proof of Proposition~\ref{uniglob}.
\end{proof}

\begin{remark}\rm
Considering the same framework \eqref{expotpot} of Remark~\ref{remuniexx}, condition~\eqref{simcond} is fulfilled for every choice of $\rho\geq 0$,
and hence large solutions in $B_1(0)$ are radially symmetric and increasing.
\end{remark}

\noindent
Next, we would like to get the radial symmetry of the solutions to \eqref{prob-ball} in the unit ball under a merely asymptotic
condition on the data $a$ and $f$ (as opposed to the global condition imposed in~\eqref{simcond}) by using
\cite[Theorem 1.1]{PV}. Throughout the rest of this section we shall assume that
$f\in C^2(\R)$ and $a\in C^2(\R)$. By direct computation, from formula~\eqref{hprimo}, there holds
\begin{align}
	\label{hsecondo}
h''(s)=&\frac{1}{2}a^{-7/2}(g(s))\Big\{2f''(g(s))a^2(g(s))-3 f'(g(s))a'(g(s))a(g(s))  \notag \\
&-f(g(s))a''(g(s))a(g(s))+2f(g(s))(a'(g(s)))^2 \Big\},\qquad\text{for every $s\in\R$}.
\end{align}
In \cite[Theorem 1.1]{PV} one of the main assumption is that the function $h$ is asymptotically convex,
namely there exists $R>0$ such that $h|_{(R,+\infty)}$ is convex. On account of formula~\eqref{hsecondo}, 
a sufficient condition for this to be the case is that
\begin{equation}
	\label{hsecondo-cond}
\liminf_{s\to+\infty}\Big\{2f''(s)a^2(s)-3 f'(s)a'(s)a(s)-f(s)a''(s)a(s)+2f(s)(a'(s))^2\Big\}>0.
\end{equation}
Hence this condition only depends on the asymptotic behavior of $a$ and $f$ and their first
and second derivatives. We shall now discuss the various situations, as already done for
the study of existence of solutions.

\subsection{Polynomial growth}

Assume that there exists $a_{\infty}>0$ such that 
for the first and second derivatives of $a$ there hold
\begin{equation}
	\label{a2lim}
\lim_{s\to +\infty}\frac{a'(s)} {s^{k-1}}=a_{\infty}k,\qquad	
\lim_{s\to +\infty}\frac{a''(s)} {s^{k-2}}=a_{\infty}k(k-1).
\end{equation}
\noindent
First observe that condition~\eqref{a2lim} implies~\eqref{1.7},
 and that for $k>1$ only the right limit in \eqref{a2lim} is needed.
We can now formulate the following symmetry results in the various situation where 
we have already established existence of large solutions.

\begin{proposition}[$f$ with exponential growth]
	\label{potexp-rad}
Assume that 
condition~\eqref{a2lim} holds and that there exist $\beta>0$ and $f_{\infty}>0$ such that
\begin{equation}
	\label{fassexppot-rad}
\lim_{s\to +\infty}\frac{f''(s)} {e^{2\beta s}}=4\beta^2 f_{\infty} .
\end{equation}
Then any solution to problem~\eqref{prob} in $B_1(0)$ is radially symmetric and increasing.
\end{proposition}
\begin{proof}
Condition~\eqref{fassexppot-rad} implies that 
$$\lim_{s\to +\infty}\frac{f(s)} {e^{2\beta s}}= f_{\infty} ,\qquad \lim_{s\to +\infty}\frac{f'(s)} {e^{2\beta s}}=2\beta f_{\infty}.$$
Concerning \eqref{hsecondo-cond}, we have
\begin{align*}
& 2f''(s)a^2(s)-3 f'(s)a'(s)a(s)-f(s)a''(s)a(s)+2f(s)(a'(s))^2  \\
&=8\beta^2 f_{\infty} e^{2\beta s}(1+o(1))a_{\infty}^2 s^{2k}(1+o(1))-6\beta f_{\infty} e^{2\beta s}(1+o(1))a_{\infty}^2k s^k s^{k-1}(1+o(1)) \\
&-f_{\infty} e^{2\beta s}(1+o(1))a_{\infty}^2 k(k-1)s^{k-2}s^k(1+o(1))+2 f_{\infty} e^{2\beta s}(1+o(1))a_{\infty}^2k^2 s^{2k-2}(1+o(1)) \\
&=f_{\infty} a_{\infty}^2 e^{2\beta s} s^{2k-2}\Big[8\beta^2 s^2(1+o(1))-6\beta k s (1+o(1))-k(k-1)+2k^2+o(1) \Big] \\
&=f_{\infty} a_{\infty}^2 e^{2\beta s} s^{2k}[8\beta^2+o(1)]>0,\quad\text{for all $s>0$ large,}
\end{align*}
concluding the proof.
\end{proof}

\begin{proposition}[$f$ with polynomial growth]
	\label{potpot-rad}
Assume that 
condition~\eqref{a2lim} holds
and that there exist $p>1$ and $ f_{\infty}>0$ such that
\begin{equation}
	\label{fasspot-rad}
\lim_{s\to +\infty}\frac{f''(s)} {s^{p-2}}= f_{\infty}p(p-1).
\end{equation}
Then, if $p>k+1$, any solution to problem~\eqref{prob} in $B_1(0)$ is radially symmetric and increasing.
\end{proposition}
\begin{proof}

First observe that condition \eqref{fasspot-rad} implies that $\lim_{s\to +\infty}\frac{f'(s)} {s^{p-1}}= f_{\infty}p$ so that \eqref{fasspotpot} holds. 
Concerning \eqref{hsecondo-cond}, we have
\begin{align*}
& 2f''(s)a^2(s)-3 f'(s)a'(s)a(s)-f(s)a''(s)a(s)+2f(s)(a'(s))^2  \\
&=2p(p-1) f_{\infty}s^{p-2}a_{\infty}^2 s^{2k}(1+o(1))-3 f_{\infty}ps^{p-1}a_{\infty}^2k s^k s^{k-1}(1+o(1)) \\
&-f_{\infty}s^p a_{\infty}^2 k(k-1)s^{k-2} s^k(1+o(1))+2 f_{\infty}s^pa_{\infty}^2k^2 s^{2k-2}(1+o(1)) \\
&=f_{\infty} a_{\infty}^2 s^{p+2k-2}\Big[2p(p-1)-3pk-k(k-1)+2k^2+o(1)\Big] \\
&=f_{\infty} a_{\infty}^2 s^{p+2k-2}\Big[2p^2-(2+3k)p+k^2+k+o(1)\Big]  \\
&=f_{\infty} a_{\infty}^2 s^{p+2k-2}\Big[(p-k-1)(2p-k)+o(1)\Big] >0
\end{align*}
for all $s>0$ large being $p>k+1$, concluding the proof.
\end{proof}

\subsection{Exponential growth}

Assume that there exist $\gamma>0$ and $a_{\infty}>0$ such that 
for the second derivative of $a$ there holds
\begin{equation}
	\label{aexplim}
\lim_{s\to +\infty}\frac{a''(s)} {e^{2\gamma s}}=4\gamma^2 a_{\infty}.
\end{equation}

\noindent
Then, we have the following 

\begin{proposition}[$f$ with exponential growth]
	\label{expexp-rad}
Assume that conditions  \eqref{fassexppot-rad} and 
\eqref{aexplim} hold. 
Then, if $\beta>\gamma$, any solution to problem~\eqref{prob} in $B_1(0)$ is radially symmetric and increasing.
\end{proposition}
\begin{proof}
First observe that condition \eqref{aexplim} implies 
$$
\lim_{s\to +\infty}\frac{a'(s)} {e^{2\gamma s}}=2\gamma a_{\infty},\qquad
\lim_{s\to +\infty}\frac{a(s)} {e^{2\gamma s}}= a_{\infty}.$$
Concerning \eqref{hsecondo-cond}, we have
\begin{align*}
& 2f''(s)a^2(s)-3 f'(s)a'(s)a(s)-f(s)a''(s)a(s)+2f(s)(a'(s))^2  \\
\noalign{\vskip3pt}
&=8\beta^2 f_{\infty} e^{2\beta s}a_{\infty}^2 e^{4\gamma s}(1+o(1))-6\beta f_{\infty} e^{2\beta s} 2\gamma a_{\infty}^2e^{2\gamma s} e^{2\gamma s} (1+o(1)) \\
\noalign{\vskip3pt}
&-f_{\infty} e^{2\beta s}4\gamma^2a_{\infty}^2 e^{4\gamma s}(1+o(1))+2f_{\infty} e^{2\beta s}4a_{\infty}^2 \gamma^2  e^{4\gamma s}(1+o(1)) \\
\noalign{\vskip1pt}
&=4 f_{\infty} a_{\infty}^2 e^{2(\beta+2\gamma) s} \big[2\beta^2-3\gamma\beta + \gamma^2 +o(1)\big]>0
\end{align*}
for all $s>0$ large if $2\beta^2-3\gamma\beta+\gamma^2=(\beta-\gamma)(2\beta-\gamma)>0$. 
Since $\beta>\gamma$, we get the desired conclusion.
\end{proof}

\subsection{Logarithmic growth}

Assume that there exist $\gamma>0$ and $a_{\infty}>0$ such that
for the first and second derivatives of $a$ there hold
\begin{equation}
	\label{aloglim}
\lim_{s\to +\infty}\frac{a'(s)s} {(\log s)^{2\gamma-1}}=2\gamma a_{\infty},\qquad	
\lim_{s\to +\infty}\frac{a''(s)s^2} {(\log s)^{2\gamma-1}}=-2\gamma a_{\infty}.
\end{equation}

\noindent
Then, we have the following 

\begin{proposition}[$f$ with exponential growth]
	\label{logexp-rad}
Assume that 
condition \eqref{fassexppot-rad} and \eqref{aloglim} 
 hold for $\beta>0$ and $f_{\infty}>0$.
Then any solution to problem~\eqref{prob} in $B_1(0)$ is radially symmetric and increasing.
\end{proposition}
\begin{proof}
First observe that \eqref{aloglim} implies~\eqref{loggrow}. Then, 
concerning \eqref{hsecondo-cond}, we have
\begin{align*}
& 2f''(s)a^2(s)-3 f'(s)a'(s)a(s)-f(s)a''(s)a(s)+2f(s)(a'(s))^2  \\
&=8\beta^2  f_{\infty}e^{2\beta s}a_{\infty}^2(\log s)^{4\gamma}(1+o(1))-12\beta f_{\infty} e^{2\beta s}\gamma a_{\infty}^2\frac{(\log s)^{2\gamma-1}}{s}(\log s)^{2\gamma}(1+o(1)) \\
&+2 f_{\infty} e^{2\beta s}  \gamma a_{\infty}^2\frac{(\log s)^{2\gamma-1}}{s^2}(\log s)^{2\gamma}(1+o(1))+8 f_{\infty} e^{2\beta s}\gamma^2a_{\infty}^2 \frac{(\log s)^{4\gamma-2}}{s^2} (1+o(1))\\
&= f_{\infty}a_{\infty}^2 e^{2\beta s}(\log s)^{4\gamma}\Big[8\beta^2-\frac{12\beta\gamma}{s\log s}+\frac{2\gamma}{s^2\log s}+\frac{8\gamma^2}{s^2(\log s)^2}+o(1)\Big]>0
\end{align*}
for all $s>0$ large, completing the proof.
\end{proof}

\noindent
Then, we have the following 

\begin{proposition}[$f$ with polynomial growth]
	\label{logpol-rad}
Assume that conditions~\eqref{fasspot-rad} and \eqref{aloglim}  hold.
Then, if $p>1$, any solution to problem~\eqref{prob} in $B_1(0)$ is radially symmetric and increasing.
\end{proposition}
\begin{proof}
Concerning \eqref{hsecondo-cond}, we have
\begin{align*}
& 2f''(s)a^2(s)-3 f'(s)a'(s)a(s)-f(s)a''(s)a(s)+2f(s)(a'(s))^2  \\
& =2 f_{\infty}p(p-1)s^{p-2}a_{\infty}^2(\log s)^{4\gamma}(1+o(1))-6 f_{\infty}p s^{p-1}\gamma a_{\infty}^2\frac{(\log s)^{2\gamma-1}}{s}(\log s)^{2\gamma}(1+o(1)) \\
& +2 f_{\infty}s^p\gamma a_{\infty}^2\frac{(\log s)^{2\gamma-1}}{s^2}(\log s)^{2\gamma}(1+o(1))+8f_{\infty} s^p\gamma^2 a_{\infty}^2\frac{(\log s)^{4\gamma-2}}{s^2}(1+o(1)) \\
&= f_{\infty} a_{\infty}^2 s^{p-2}(\log s)^{4\gamma}\Big[2p(p-1)-\frac{6p\gamma}{\log s}+\frac{2\gamma}{\log s}+\frac{8\gamma^2}{(\log s)^2}+o(1) \Big]>0
\end{align*}
for all $s>0$ large, completing the proof.
\end{proof}

\section{Blow-up rate of solutions}
\noindent
In this section we shall provide the proof of Theorems~\ref{blowI} and~\ref{blowII}.

\subsection{Proof of Theorem~\ref{blowI}}

\noindent
Let $\Omega$ be a bounded domain of $\R^N$ which satisfies an inner and an outer sphere
condition at each point of the boundary $\partial\Omega$. Consider the following condition
\begin{equation}
	\label{con-blowup1}
\lim_{u\to+\infty}\sqrt{F(g(u))}\int_u^{+\infty}\frac{\int_0^t \sqrt{F(g(s))}ds}{F^{3/2}(g(t))}dt=0,
\end{equation}
which merely depends upon the asymptotic behavior of $F$ and $a$. Then, assuming that
condition {\bf E}
holds, by directly applying
\cite[Theorem 1.10]{CD} to the semi-linear problem~\eqref{prob-semi}, if $\eta$ denotes the unique solution to 
$$
\eta'=-\sqrt{2F\circ g\circ\eta},\qquad \lim_{t\to 0^+}\eta(t)=+\infty,
$$
it follows that any blow-up solution $v\in C^2(\Omega)$ 
to \eqref{prob-semi} satisfies
\begin{equation}
	\label{vexpans}
v(x)=\eta({\rm d}(x,\partial\Omega))+o(1),\qquad\text{as ${\rm d}(x,\partial\Omega)\to 0$}
\end{equation}
if and only if \eqref{con-blowup1} holds.
By virtue of \eqref{cauchy} and the asymptotic behavior of $a(s)$ for $s$ large (i.e. \eqref{1.7}, \eqref{exgrow} or \eqref{loggrow}),
it is readily verified that there exists a positive constant $L$ such that
\begin{equation}\label{lipschitz-g}
|g(\tau_2)-g(\tau_1)|\leq L\frac{|\tau_2-\tau_1|}{a^{\frac{1}{2}}(\min\{g(\tau_1),g(\tau_2)\})},\qquad
\text{for every $\tau_1,\tau_2>0$ large}.
\end{equation}
Therefore, 
under assumption \eqref{con-blowup1}, any blow-up solution $u\in C^2(\Omega)$ 
to the quasi-linear problem \eqref{prob} satisfies
\begin{equation*}
|u(x)-g(\eta({\rm d}(x,\partial\Omega)))|=|g(v(x))-g(\eta({\rm d}(x,\partial\Omega)))|\leq L \frac{|v(x)-\eta({\rm d}(x,\partial\Omega))|}{a^{\frac{1}{2}}(\min\{u(x), g(\eta({\rm d}(x,\partial\Omega)))\})}
\end{equation*}
as ${\rm d}(x,\partial\Omega)\to 0$, namely due to \eqref{vexpans}
\begin{equation}
	\label{blowupconcl-u}
u(x)=g\circ\eta({\rm d}(x,\partial\Omega))+\frac 1 {a^{\frac{1}{2}}(\min\{u(x), g(\eta({\rm d}(x,\partial\Omega)))\})}o(1),\qquad\text{as $x\to x_0\in\partial\Omega$.}
\end{equation}
Moreover, in case \eqref{1.7} holds, then Lemma~\ref{l2.3} implies there exists a positive constant $C$ such that
$$ a^{\frac{1}{2}}(\min\{u(x), g(\eta({\rm d}(x,\partial\Omega)))\})\geq C \min\{ u(x)^{\frac k2}, \eta^{\frac k{k+2}}({\rm d}(x,\partial\Omega))\},$$
as ${\rm d}(x,\partial\Omega)\to 0$, while, if \eqref{exgrow} is satisfied, from Lemma~\ref{growl} we have that there exists a positive constant $C$ such that 
$$ a^{\frac{1}{2}}(\min\{u(x), g(\eta({\rm d}(x,\partial\Omega)))\})\geq C  \min\{e^{\gamma u(x)}, e^{\gamma g(\eta({\rm d}(x,\partial\Omega)))}\}\geq  \min\{e^{\gamma u(x)},\eta^{\frac 1{\alpha}}({\rm d}(x,\partial\Omega))\}$$
as ${\rm d}(x,\partial\Omega)\to 0$, for any $\alpha>1$.
Using \eqref{blowupconcl-u} we get Theorem~\ref{blowI} once~\eqref{con-blowup1} is satisfied.

\subsection{Polynomial growth}

\noindent
We have the following 
\begin{proposition}
\label{blow-polpol}
Assume that there exist $k>0$, $a_{\infty}>0$, $p>k+1$ and $f_{\infty}>0$ such that
\begin{equation*}
\lim_{s\to +\infty}\frac{a(s)} {s^k}=a_{\infty},\qquad
\lim_{s\to +\infty}\frac{f(s)} {s^p}=f_{\infty}.
\end{equation*}
Then condition~\eqref{con-blowup1} is fulfilled if and only if $p>2k+3$. When $p<2k+3$ it holds
\begin{equation*}
\lim_{u\to+\infty}\sqrt{F(g(u))}\int_u^{+\infty}\frac{\int_0^t \sqrt{F(g(s))}ds}{F^{3/2}(g(t))}dt=+\infty.
\end{equation*}
\end{proposition}
\begin{proof}
We denote by $C$ a positive constant and by $C'$ a constant without any sign
restriction, which may vary from one place to another.
By the proof of Proposition~\ref{potpot}, we learn that there exists a constant $R>0$ such that 
\begin{equation}
	\label{eunapot}
Cs^{\frac{p+1}{k+2}}\leq \sqrt{F(g(s))} \leq 2Cs^{\frac{p+1}{k+2}},  \,\,\,\quad\text{for every $s\geq R$.}
\end{equation}
\noindent In turn, for every $t\geq R$, we obtain
\begin{equation}
	\label{eduepot}
C t^{\frac{k+p+3}{k+2}}+C'\leq \int_0^t \sqrt{F(g(s))}ds\leq C+\int_{R}^t \sqrt{F(g(s))}ds\leq C'+Ct^{\frac{k+p+3}{k+2}},
\end{equation}
Furthermore, by~\eqref{eunapot}-\eqref{eduepot}, we get
\begin{align*}
&\limsup_{u\to+\infty}\sqrt{F(g(u))}\int_u^{+\infty}\frac{\int_0^t \sqrt{F(g(s))}ds}{F^{3/2}(g(t))}dt  \\
&\leq C\limsup_{u\to+\infty}u^{\frac{p+1}{k+2}}\int_u^{+\infty}\frac{C'+C t^{\frac{k+p+3}{k+2}}}{t^{\frac{3p+3}{k+2}}}dt \\
& \leq C \limsup_{u\to+\infty}u^{\frac{p+1}{k+2}}\int_u^{+\infty} t^{\frac{k-2p}{k+2}}dt
= C\limsup_{u\to+\infty}u^{\frac{2k+3-p}{k+2}}=0,
\end{align*}
yielding~\eqref{con-blowup1}. Assume now instead that $k+1<p\leq 2k+3$, we have 
\begin{align*}
&\liminf_{u\to+\infty}\sqrt{F(g(u))}\int_u^{+\infty}\frac{\int_0^t \sqrt{F(g(s))}ds}{F^{3/2}(g(t))}dt  \\
&\geq C\liminf_{u\to+\infty}u^{\frac{p+1}{k+2}}\int_u^{+\infty}\frac{C'+Ct^{\frac{k+p+3}{k+2}}}{t^{\frac{3p+3}{k+2}}}dt \\
& \geq C \liminf_{u\to+\infty}u^{\frac{p+1}{k+2}}\int_u^{+\infty} \big(C't^{-\frac{3p+3}{k+2}}+Ct^{\frac{k-2p}{k+2}}\big)dt \\
&= \liminf_{u\to+\infty}\big(C'u^{\frac{k-2p}{k+2}}+ Cu^{\frac{2k+3-p}{k+2}}\big)=
\begin{cases}
+\infty, & \text{for $p<2k+3$} \\
C, &  \text{for $p=2k+3$},
\end{cases}
\end{align*}
being $\frac{k-2p}{k+2}<-1$. This concludes the proof.
\end{proof}

\noindent
We have the following 
\begin{proposition}
\label{blow-polexp}
Assume that there exist $k>0$, $a_{\infty}>0$, $\beta>0$ and $f_{\infty}>0$ such that
\begin{equation*}
\lim_{s\to +\infty}\frac{a(s)} {s^k}=a_{\infty},\qquad
\lim_{s\to +\infty}\frac{f(s)} {e^{2\beta s}}=f_{\infty}.
\end{equation*}
Then condition~\eqref{con-blowup1} is always fulfilled. 
\end{proposition}
\begin{proof}
We denote by $C$ a positive constant and by $C'$ a constant without any sign
restriction, which may vary from one one place to another.
By the proof of Proposition~\ref{potexp}, we learn that for any fixed 
$\eps$ there exists a constant $R=R(\eps)>0$ such that 
\begin{equation}
	\label{eunapotex}
\sqrt{F(g(s))} \geq  e^{(\beta g_{\infty}-\eps) s^{\frac{2}{k+2}}},  \,\,\,\quad\text{for every $s\geq R$.}
\end{equation}
Moreover, for every $\eps>0$ fixed we have
\begin{equation*}
\lim_{s\to+\infty}\frac{F(g(s))}{e^{(2\beta g_{\infty}+2\eps) s^{\frac{2}{k+2}}}}
=\frac{f_{\infty}}{2\beta}\lim_{s\to+\infty}\frac{e^{2\beta g_{\infty}s^{\frac{2}{k+2}}(1+o(1)) }}{e^{(2\beta g_{\infty}+2\eps) s^{\frac{2}{k+2}}}}
=\frac{f_{\infty}}{2\beta}\lim_{s\to+\infty}\frac{e^{2\beta g_{\infty} o(1)s^{\frac{2}{k+2}}}}{e^{2\eps s^{\frac{2}{k+2}}}}=0.
\end{equation*}
In turn, 
increasing $R$ if needed, we have
\begin{equation}
	\label{eduepotex}
 e^{(\beta g_{\infty}-\eps) s^{\frac{2}{k+2}}}\leq \sqrt{F(g(s))} \leq e^{(\beta g_{\infty}+\eps) s^{\frac{2}{k+2}}},\qquad\text{for every $s\geq R$}.
\end{equation}
Choose now $\bar\eps>0$ such that $7\bar\eps<\beta g_{\infty}$ and let $\bar R\geq R$ be such that
$$
w^{k/2}\leq e^{\bar\eps w},\qquad \text{for every $w\geq \bar R^{\frac{2}{k+2}}$}.
$$ 
Then, in light of inequality~\eqref{eduepotex}, we obtain
\begin{align*}
\int_0^t \sqrt{F(g(s))}ds& =\int_0^{\bar R} \sqrt{F(g(s))}ds+\int_{\bar R}^t \sqrt{F(g(s))}ds 
\leq C+ \int_{\bar R}^t e^{(\beta g_{\infty}+\bar\eps) s^{\frac{2}{k+2}}}ds \\
&= C+ C\int_{\bar R^{\frac{2}{k+2}}}^{t^{\frac{2}{k+2}}} e^{(\beta g_{\infty}+\bar\eps) w} w^{k/2}dw
\leq C+ C\int_{\bar R^{\frac{2}{k+2}}}^{t^{\frac{2}{k+2}}} e^{(\beta g_{\infty}+2\bar\eps) w}dw \\
&=C'+ C e^{(\beta g_{\infty}+2\bar\eps) t^{\frac{2}{k+2}}}\leq C e^{(\beta g_{\infty}+2\bar\eps) t^{\frac{2}{k+2}}},
\end{align*}
for every $t$ large. Therefore,
\begin{align*}
\lim_{u\to+\infty}&\sqrt{F(g(u))}\int_u^{+\infty}\frac{\int_0^t \sqrt{F(g(s))}ds}{F^{3/2}(g(t))}dt  \\
&\leq C\lim_{u\to+\infty}e^{(\beta g_{\infty}+\bar\eps) 
u^{\frac{2}{k+2}}}\int_u^{+\infty}\frac{e^{(\beta g_{\infty}+2\bar\eps) t^{\frac{2}{k+2}}}}{e^{(3\beta g_{\infty}-3\bar \eps) t^{\frac{2}{k+2}}}}dt \\
&= C\lim_{u\to+\infty}e^{(\beta g_{\infty}+\bar\eps) u^{\frac{2}{k+2}}}\int_{u^{\frac{2}{k+2}}}^{+\infty} e^{(-2\beta g_{\infty}+5\bar\eps)w} w^{k/2} dw \\
&\leq C\lim_{s\to+\infty}e^{(\beta g_{\infty}+\bar\eps) s}\int_{s}^{+\infty} e^{(-2\beta g_{\infty}+6\bar\eps)w}  dw \\
&\leq C\lim_{s\to+\infty}e^{(\beta g_{\infty}+\bar\eps) s}e^{(-2\beta g_{\infty}+6\bar\eps)s}  \\
& =C\lim_{s\to+\infty}e^{-(\beta g_{\infty}-7\bar\eps)s}=0, 
\end{align*}
yielding~\eqref{con-blowup1}. This concludes the proof.
\end{proof}

\subsection{Exponential growth}

\noindent
We have the following 
\begin{proposition}
\label{blow-expexp}
Assume that there exist $\gamma>0$, $a_{\infty}>0$, $\beta>\gamma$ and $ f_{\infty}>0$ such that
\begin{equation*}
\lim_{s\to +\infty}\frac{a(s)} {e^{2\gamma s}}=a_{\infty},\qquad
\lim_{s\to +\infty}\frac{f(s)} {e^{2\beta s}}=f_{\infty}.
\end{equation*}
Then condition~\eqref{con-blowup1} is fulfilled provided that $\beta>2\gamma$. When $\beta<2\gamma$ it holds
\begin{equation*}
\lim_{u\to+\infty}\sqrt{F(g(u))}\int_u^{+\infty}\frac{\int_0^t \sqrt{F(g(s))}ds}{F^{3/2}(g(t))}dt=+\infty.
\end{equation*}
\end{proposition}
\begin{proof}
We denote by $C$ a positive constant and by $C'$ a constant without any sign
restriction, which may vary from one place to another.
By the proof of Proposition~\ref{expexp}, we learn that for every $\eps>0$ there exists a positive value $R=R(\eps)$ large enough that
$F(g(s)) \geq s^{2\beta/\gamma-2\eps}$ for every $s\geq R$.
Furthermore, enlarging $R$ if needed, we have
$$
F(g(s)) \leq s^{\frac{2\beta}{\gamma}+2\eps},  \,\,\,\quad\text{for every $s\geq R$.}
$$
Assume that $\beta>2\gamma$. Choosing now $\bar\eps>0$ so small that $2-\frac{\beta}{\gamma}+5\bar\eps<0$, we have
\begin{equation}
	\label{euna}
  s^{\frac{2\beta}{\gamma}-2\bar \eps}\leq F(g(s)) \leq s^{\frac{2\beta}{\gamma}+2\bar\eps},  \,\,\,\quad\text{for every $s\geq R$.}
\end{equation}
In turn, for every $t\geq R$, we obtain
\begin{equation}
	\label{edue}
C t^{1+\frac{\beta}{\gamma}-\bar \eps}+C'\leq \int_0^t \sqrt{F(g(s))}ds\leq C+
\int_{R}^t \sqrt{F(g(s))}ds\leq C'+C t^{1+\frac{\beta}{\gamma}+{\bar\eps}}\leq Ct^{1+\frac{\beta}{\gamma}+{\bar\eps}} .
\end{equation}
In turn, by~\eqref{euna}-\eqref{edue}, we get
\begin{align*}
&\lim_{u\to+\infty}\sqrt{F(g(u))}\int_u^{+\infty}\frac{\int_0^t \sqrt{F(g(s))}ds}{F^{3/2}(g(t))}dt  \\
&\leq
C\lim_{u\to+\infty}u^{\frac{\beta}{\gamma}+{\bar\eps}}\int_u^{+\infty}\frac{ t^{1+\frac{\beta}{\gamma}+{\bar\eps}}}{t^{3\beta/\gamma-3\bar \eps}}dt \\
& \leq C \lim_{u\to+\infty}u^{\frac{\beta}{\gamma}+{\bar\eps}}\int_u^{+\infty} t^{1-\frac{2\beta}{\gamma}+4{\bar\eps}}dt
= C \lim_{u\to+\infty}u^{2-\frac{\beta}{\gamma}+5\bar\eps}=0,
\end{align*}
yielding~\eqref{con-blowup1}.
Assume now instead that $\beta<2\gamma$. Choosing now $\bar\eps>0$ so small that $2-\frac{\beta}{\gamma}- 5\bar\eps >0$, we have
\begin{align*}
&\lim_{u\to+\infty}\sqrt{F(g(u))}\int_u^{+\infty}\frac{\int_0^t \sqrt{F(g(s))}ds}{F^{3/2}(g(t))}dt  \\
&\geq \lim_{u\to+\infty}u^{\frac{\beta}{\gamma}-\bar \eps}\int_u^{+\infty}\frac{C'+Ct^{1+\frac{\beta}{\gamma}-\bar \eps}}{t^{3\beta/\gamma+{3\bar\eps}}}dt \\
& \geq  \lim_{u\to+\infty}u^{\frac{\beta}{\gamma}-\bar \eps}\int_u^{+\infty} C't^{-\frac{3\beta}{\gamma}-{3\bar\eps}}+Ct^{1-\frac{2\beta}{\gamma}-4\bar\eps}dt \\
&= \lim_{u\to+\infty}\big( C'u^{1-\frac{2\beta}{\gamma}-4\bar\eps}+Cu^{2-\frac{\beta}{\gamma}-5\bar\eps}\big)=+\infty,
\end{align*}
concluding the proof since $1-\frac{2\beta}{\gamma}-4\bar\eps <0$ and $2-\frac{\beta}{\gamma}-5\bar \eps>0$.
\end{proof}

\subsection{Logarithmic growth}

\noindent
We have the following 
\begin{proposition}
\label{blow-logpot}
Assume that there exist $\gamma>0$, $a_{\infty}>0$, $p>1$ and $f_{\infty}>0$ such that
\begin{equation*}
\lim_{s\to +\infty}\frac{a(s)} {(\log s)^{2\gamma}}=a_{\infty},\qquad
\lim_{s\to +\infty}\frac{f(s)} {s^p}=f_{\infty}.
\end{equation*}
Then condition~\eqref{con-blowup1} is fulfilled provided that $p>3$. When $1<p<3$ it holds
\begin{equation*}
\lim_{u\to+\infty}\sqrt{F(g(u))}\int_u^{+\infty}\frac{\int_0^t \sqrt{F(g(s))}ds}{F^{3/2}(g(t))}dt=+\infty.
\end{equation*}
\end{proposition}
\begin{proof}
We denote by $C$ a positive constant and by $C'$ a constant without any sign
restriction, which may vary from one place to another.
Taking into account Lemma \ref{growlog},
for every $\eps\in (0,1)$ there exists $R=R(\eps)>0$ such that
$$
C s^{(p+1)(1-\eps)}\leq F(g(s))\leq Cs^{p+1},\qquad\text{for every $s\geq R$.}
$$
So for every $\eps\in (0,1)$ and for all $t\geq R$, we obtain
\begin{align*}
Ct^{\frac{p+1}2(1-\eps)+1}+C'\leq  \int_0^t \sqrt{F(g(s))}ds &=\int_0^{R} \sqrt{F(g(s))}ds+\int_{R}^t \sqrt{F(g(s))}ds	\\
&\leq C+C\int_{R_{\bar\eps}}^t s^{\frac{p+1}{2}}ds\leq C t^{\frac{p+3}{2}}.
\end{align*}	
Assume that $p>3$ and let $\bar\eps>0$ with $p-3-3\bar\eps(p+1)>0$.
Whence, we get
\begin{align*}
\limsup_{u\to+\infty}\sqrt{F(g(u))}&\int_u^{+\infty}\frac{\int_0^t \sqrt{F(g(s))}ds}{F^{3/2}(g(t))}dt  \\
&\leq C\limsup_{u\to+\infty}	u^{\frac{p+1}{2}}\int_u^{+\infty}\frac{t^{\frac{p+3}{2}}}{t^{\frac{3p+3}{2}(1-\eps)}}dt \\
& =C\limsup_{u\to+\infty} u^{\frac{p+1}{2}}\int_u^{+\infty}t^{-p+\frac{3\bar\eps}{2}(p+1)} dt\\
& =C\limsup_{u\to+\infty} u^{-\frac{p-3-3\bar\eps(p+1)}{2}}=0.
\end{align*}	
On the contrary, if $1<p<3$, fix $\bar\eps$ so small that $\frac{3-p}{2}-\bar\eps(p+1)>0$. In turn, we deduce 
\begin{align*}
\liminf_{u\to+\infty}\sqrt{F(g(u))}&\int_u^{+\infty}\frac{\int_0^t \sqrt{F(g(s))}ds}{F^{3/2}(g(t))}dt  \\
& \geq E\liminf_{u\to+\infty}	u^{\frac{p+1}{2}(1-\bar\eps)}\int_u^{+\infty}\frac{t^{\frac{p+1}{2}(1-\bar\eps)+1}+C'}{t^{\frac{3p+3}{2}}}  dt\\
& =\lim_{u\to+\infty} u^{\frac{p+1}{2}(1-\bar\eps)}\big(C u^{1-p-\bar\eps\frac{p+1}{2}}+C'u^{1-\frac{3}{2}(p+1)}  \big)   \\
& =\lim_{u\to+\infty} \big(C u^{\frac{3-p}{2}-\bar\eps(p+1)}+C' u^{-p-\bar\eps\frac{p+1}{2}}\big)=+\infty.                      
\end{align*}
This concludes the proof.
\end{proof}

\noindent
We have the following 
\begin{proposition}
\label{blow-logexp}
Assume that there exist $\gamma>0$, $a_{\infty}>0$, $\beta>0$ and $f_{\infty}>0$ such that
\begin{equation*}
\lim_{s\to +\infty}\frac{a(s)} {(\log s)^{2\gamma}}=a_{\infty},\qquad
\lim_{s\to +\infty}\frac{f(s)} {e^{2\beta s}}=f_{\infty}.
\end{equation*}
Then condition~\eqref{con-blowup1} is fulfilled. 
\end{proposition}
\begin{proof}
We denote by $C$ a positive constant and by $C'$ a constant without any sign
restriction, which may vary from one place to another.
For every $\eps\in (0,1)$ there exists $R=R(\eps)>0$  such that
$$
C e^{2\beta s^{1-\eps}}\leq F(g(s))\leq C e^{2\beta s},\qquad\text{for every $s\geq R$.}
$$
Observe that 
$$
\lim_{u\to+\infty} \int_u^{+\infty}\frac{\int_0^t \sqrt{F(g(s))}ds}{F^{3/2}(g(t))}dt
=\lim_{u\to+\infty} \int_1^{+\infty}\frac{\int_0^t \sqrt{F(g(s))}ds}{F^{3/2}(g(t))}\chi_{(u,+\infty)}dt=0.
$$
In fact, the integrand belongs to $L^1(1,+\infty)$ since, for $R$ big enough and all $t$ large, we have
$$
\frac{\int_0^t \sqrt{F(g(s))}ds}{(F(g(t)))^{3/2}}\leq \frac{C+(t-R)\sqrt{F(g(t))}}{(F(g(t)))^{3/2}}
\leq C\frac{t}{F(g(t))}\leq C\frac{t}{e^{2\beta t^{1-\eps}}}.
$$ 
Then, by virtue of l'H\v opital rule, we get
\begin{align*}
\lim_{u\to+\infty}\sqrt{F(g(u))}&\int_u^{+\infty}\frac{\int_0^t \sqrt{F(g(s))}ds}{F^{3/2}(g(t))}dt \\
&=\lim_{u\to+\infty}\frac{-F^{-3/2}(g(u))\int_0^u \sqrt{F(g(s))}ds}{-\frac{1}{2}F^{-3/2}(g(u)) f(g(u))g'(u)} \\
&=2\lim_{u\to+\infty}\frac{\int_0^u \sqrt{F(g(s))}ds}{f(g(u))g'(u)}= 0.
\end{align*}
We thus only need to justify this last limit. Observe that, if $\eps\in (0,1)$, from $s^{1-\eps}\leq g(s)$ for $s$ large enough (see Lemma \ref{growlog}), we get
$g^{-1}(s)\leq s^{1/(1-\eps)}$. Then, for $R$ large enough, we have
\begin{align*}
0\leq \limsup_{u\to+\infty}\frac{\int_0^u \sqrt{F(g(s))}ds}{f(g(u))g'(u)}&\leq 
\limsup_{u\to+\infty}\frac{\sqrt{a(g(u))}(C+(u-R)\sqrt{F(g(u))})}{f(g(u))} \\
& \leq C \limsup_{u\to+\infty}\frac{u\sqrt{a(g(u))}\sqrt{F(g(u))}}{f(g(u))} \\
& = C \limsup_{t\to+\infty}\frac{g^{-1}(t)\sqrt{a(t)}\sqrt{F(t)}}{f(t)} \\
& \leq C\limsup_{t\to+\infty}\frac{t^{\frac{1}{1-\eps}}\sqrt{a(t)}\sqrt{F(t)}}{f(t)} \\
&=\frac{\sqrt{ f_{\infty}a_{\infty}}}{\sqrt{2\beta}}\lim_{t\to+\infty}\frac{t^{\frac{1}{1-\eps}}(\log t)^\gamma e^{\beta t}}{f_{\infty}e^{2\beta t}}=0.
\end{align*}
This concludes the proof.
\end{proof}

\subsection{Proof of Theorem~\ref{blowII}}

\noindent
Suppose that \eqref{segnomonot} hold with $R=0$ and that the existence conditions of Theorem \ref{main-ex} are satisfied. As before, $\eta$ denotes the unique solution to \eqref{ODEeta}. 
We consider the 
following notations, where $F(t)=\int_0^t f(\sigma)d\sigma$ and
\begin{align*}
& \psi(t):=\int_t^{+\infty}\frac{ds}{\sqrt{2F(g(s))}}, \qquad\,\,\,
\Lambda(t):=\frac{\int_{0}^t\sqrt{2F(g(s))}ds}{F(g(t))}, \quad t>0, \\
& J(t):=\frac{N-1}{2}\int_{0}^t\Lambda(\eta(s))ds,  \quad\,\,
B(t):=\frac{f(g(t))g'(t)}{\sqrt{2F(g(t))}}, \quad t>0.
\end{align*}
Notice that our condition {\bf E} holds for every choice of $r>0$ and
$$
\psi(t)<+\infty,\,\,\,\,\,\forall t>0
\,\,\quad\Longleftrightarrow \quad\,\,
\int_{g^{-1}(r)}^{+\infty}\frac{ds}{\sqrt{F_r(g(s))}}<+\infty,\,\,\,\,\,\forall r>0,\quad F_r(t)=\int_r^t f(\sigma)d\sigma,
$$
justifying the finiteness of $\psi(t)$ at each $t>0$.

In the following, we shall denote by $\sigma(x)$ the orthogonal projection on the boundary
$\partial\Omega$ of a given point $x\in\Omega$. Moreover, we shall indicate by ${\mathcal H}:\partial\Omega\to\R$ 
the mean curvature of $\partial\Omega$ (see \cite{thorpe} for a definition of mean curvature). In particular, the function $x\mapsto {\mathcal H}(\sigma(x))$
is well defined on $\Omega$. We can state the following
\begin{proposition}\label{fin1}
	Let $\Omega$ be a bounded domain of $\R^N$ of class $C^4$
	and assume that \eqref{segnomonot} hold with $R=0$ and that one of the existence conditions of Theorem \ref{main-ex} is satisfied. 
	Let us set
	$$
	\T(x):=\frac{\eta({\rm d}(x,\partial\Omega))}{ a^{\frac 12}(\min\{u(x),
	g(\eta({\rm d} (x,\partial\Omega)-{\mathcal H} (\sigma(x))J({\rm d}(x,\partial\Omega))))\})},
	\qquad x\in\Omega,$$
where $\sigma(x)$ denotes 
the projection on $\partial\Omega$ of a $x\in\Omega$ and ${\mathcal H}$ is the mean curvature of $\partial\Omega$.
	Then there exists a positive constant $L$ such that
	\begin{equation*}
	|u(x)-g\circ\eta({\rm d}(x,\partial\Omega)-{\mathcal H}(\sigma(x))J({\rm d}(x,\partial\Omega)))|\leq L\T(x)o({\rm d}(x,\partial\Omega)),
	\end{equation*}
	whenever ${\rm d}(x,\partial\Omega)$ goes to zero if the following conditions hold
\begin{align}
& \liminf_{t\to+\infty}\frac{\psi(\nu t)}{\psi(t)}>1,\quad\text{for all $\nu\in (0,1)$}, \label{primacurv}\\
& \lim_{\delta\to 0} \frac{B(\eta(\delta(1+o(1))))}{B(\eta(\delta))}=1, \label{secondacurv} \\
& \limsup_{t\to+\infty} B(t)\Lambda(t)<+\infty.    \label{terzacurv}
\end{align}
\end{proposition}
\begin{proof}
Assuming that $\Omega$ is a domain of class $C^4$, by using the main result of \cite{curvature} due to Bandle and Marcus,
if the problem $\Delta v=h(v)$ with $h(s)=f(g(s))a^{-1/2}(g(s))$ positive  
and nondecreasing on $(0,+\infty)$, satisfying the Keller-Osserman  condition and \eqref{primacurv}, \eqref{secondacurv} and \eqref{terzacurv}
then it follows
\begin{align}
|v(x)-\eta({\rm d}(x,\partial\Omega)-{\mathcal H}(\sigma(x))J({\rm d}(x,\partial\Omega)))| 
&\leq \eta({\rm d}(x,\partial\Omega))o({\rm d}(x,\partial\Omega)) ,
\label{vsvilup}
\end{align}
provided that ${\rm d}(x,\partial\Omega)$ goes to zero. 
The proof then follows from \eqref{lipschitz-g} and \eqref{vsvilup}.
\end{proof}
\vskip5pt
\noindent
In the particular case where~\eqref{caso:potpot} is satisfied with $p>k+1$, we have the following

\begin{proposition}
\label{precise-potpot}
Let $\Omega$ be a bounded domain of $\R^N$ which satisfies an inner and an outer sphere
condition at each point of the boundary $\partial\Omega$. Therefore, if \eqref{caso:potpot} hold 
with $p>2k+3$, any solution $u\in C^2(\Omega)$ to~\eqref{prob} satisfies
$$
u(x)=\frac{\Gamma}{({\rm d}(x,\partial\Omega))^{\frac{2}{p-k-1}}}(1+o(1)),\quad\,\,
\Gamma:=\left[\frac{p-k-1}{\sqrt{2(p+1)}} \frac{\sqrt{f_{\infty}}}{\sqrt {a_{\infty}}}\right]^{\frac{2}{k+1-p}}>0,
$$
whenever $x$ approaches the boundary $\partial\Omega$.
\end{proposition}
\begin{proof}
Let $\eta$ denote the unique solution to 
$$
\eta'=-\sqrt{2F\circ g\circ\eta},\qquad \lim_{t\to 0^+}\eta(t)=+\infty.
$$
Let us prove that
\begin{equation*}
\lim_{t\to 0^+}\frac{\eta(t)}{t^{\frac{k+2}{k+1-p}}}=\Gamma_0,\qquad
\Gamma_0=\left[\frac{p-k-1}{k+2}\frac{\sqrt{2f_{\infty}}}{\sqrt{p+1}}\left(\frac{k+2}{2\sqrt{a_{\infty}}}\right)^{\frac{p+1}{k+2}}\right]^{\frac{k+2}{k+1-p}}.
\end{equation*}
For any $t>0$ sufficiently close to $0$ we have $2F(g(\eta(t)))>0$ and, from
$\frac{\eta'}{\sqrt{2F\circ g\circ\eta}}=-1$, 
\begin{equation*}
\int_{\eta(t)}^{+\infty}\frac{d\xi}{\sqrt{2F(g(\xi))}}=\int_t^0\frac{\eta'(s)}{\sqrt{2F(g(\eta(s)))}}=t.
\end{equation*}
Furthermore, from \eqref{caso:potpot} and \eqref{1.8}, we have
$$
\lim_{t\to 0^+}\frac{\sqrt{2F(g(\eta(t)))}}{\eta(t)^{\frac{p+1}{k+2}}}=
\lim_{s\to+\infty}\frac{\sqrt{2F(g(s))}}{s^{\frac{p+1}{k+2}}}=\frac{\sqrt{2f_{\infty}}g_{\infty}^{\frac{p+1}{2}}}{\sqrt{p+1}},
$$
where $g_{\infty}$ was introduced in Lemma~\ref{l2.3}. Whence, recalling that $p>k+1$, we get
\begin{align}
	\label{limetaa}
\lim_{t\to 0^+}\frac{\eta(t)}{t^{\frac{k+2}{k+1-p}}}&=\lim_{t\to 0^+}\left[\frac{\eta(t)^{\frac{k+1-p}{k+2}}}{t}\right]^{\frac{k+2}{k+1-p}}
=\lim_{t\to 0^+}\left[\frac{\eta(t)^{\frac{k+1-p}{k+2}}}{\int_{\eta(t)}^{+\infty}\frac{d\xi}{\sqrt{2F(g(\xi))}}}\right]^{\frac{k+2}{k+1-p}}  \\
&=\lim_{t\to 0^+}\left[\frac{\frac{k+1-p}{k+2}\eta(t)^{\frac{-p-1}{k+2}}\eta'(t)}{
-\frac{1}{\sqrt{2F(g(\eta(t)))}}\eta'(t)}\right]^{\frac{k+2}{k+1-p}} \notag\\
&=\lim_{t\to 0^+}\left[\frac{\frac{p-k-1}{k+2}\sqrt{2F(g(\eta(t)))}}{\eta(t)^{\frac{p+1}{k+2}}}\right]^{\frac{k+2}{k+1-p}}=\Gamma_0. \notag
\end{align}
Taking into account Lemma~\ref{l2.3}, we thus obtain
\begin{equation*}
g\circ\eta({\rm d}(x,\partial\Omega))=
g_{\infty}\Gamma_0^{\frac 2{k+2}}({\rm d}(x,\partial\Omega))^{\frac{2}{k+1-p}}(1+o(1)), \quad \text{as ${\rm d}(x,\partial\Omega)\to 0$}.
\end{equation*}
Since $p>2k+3$ by virtue of \eqref{blowupconcl-u} it holds $u(x)=g\circ\eta({\rm d}(x,\partial\Omega))+o(1)$
as $x$ approaches the boundary $\partial\Omega$. Combining these equations we get the assertion. 
\end{proof}
\vskip5pt
\noindent

\begin{proposition}
	\label{blowII-prop}
	Let $\Omega$ be a bounded domain of $\R^N$ of class $C^4$,
	assume that \eqref{caso:potpot} hold with $p>k+1$ and that \eqref{segnomonot} are satisfied with $R=0$. Then the following facts hold 
	\begin{enumerate}
        \item There exists a positive constant $L$ such that
	\begin{equation*}
	|u(x)-g\circ\eta({\rm d}(x,\partial\Omega)-{\mathcal H}(\sigma(x))J({\rm d}(x,\partial\Omega)))|\leq L\T(x)o({\rm d}(x,\partial\Omega)),
	\end{equation*}
	whenever ${\rm d}(x,\partial\Omega)$ goes to zero, where
        $$\T(x):=         \frac{({\rm d}(x,\partial\Omega))^{\frac{k+2}{k+1-p}}}{\min\{u^{k/2}(x)),
	({\rm d}(x,\partial\Omega)-{\mathcal H}(\sigma(x))J({\rm d}(x,\partial\Omega)))^{k/(k+1-p)}\}},
	\qquad x\in\Omega,$$
        where $\sigma(x)$ denotes 
the projection on $\partial\Omega$ of a $x\in\Omega$ and ${\mathcal H}$ is the mean curvature of $\partial\Omega$.
	\item If $k+3<p\leq 2k+3$, then 
	$$
	u(x)=\Gamma\frac{1}{({\rm d}(x,\partial\Omega))^{\frac{2}{p-k-1}}}(1+o(1))
	+\Gamma' {\mathcal H}(\sigma(x))({\rm d}(x,\partial\Omega))^{\frac{p-k-3}{p-k-1}}(1+o(1)),
	$$
	whenever $x$ approaches $\partial\Omega$, where $\Gamma$ and $\Gamma '$ are as defined in \eqref{defGamma}.
	\vskip4pt
	\item If $p\leq k+3$, then 
	$$
	u(x)=\Gamma\frac{1}{({\rm d}(x,\partial\Omega))^{\frac{2}{p-k-1}}}(1+o(1))
	+\Gamma'\frac{{\mathcal H}(\sigma(x)) }{({\rm d}(x,\partial\Omega))^{\frac{3+k-p}{p-k-1}}}(1+o(1)),
	$$
	whenever $x$ approaches $\partial\Omega$.
\end{enumerate}
\end{proposition}
\begin{proof}
We want to apply Proposition \ref{fin1}.
Let us check that condition~\eqref{primacurv} holds. 
From the validity of the Keller-Osserman condition and the definition of $\psi$, we have
$\psi(\nu t)\to 0$ as $t\to+\infty$, for every $\nu\in (0,1]$. Given $\nu\in (0,1)$, using \eqref{caso:potpot} and \eqref{1.8}, we have in turn that
\begin{align*}
& \lim_{t\to+\infty}\frac{\psi(\nu t)}{\psi(t)}
=\lim_{t\to+\infty}\frac{\int_{\nu t}^{+\infty}\frac{ds}{\sqrt{2F(g(s))}}}{\int_t^{+\infty}\frac{ds}{\sqrt{2F(g(s))}}} 
 = \nu\sqrt{\lim_{t\to+\infty}\frac{F(g(t))}{F(g(\nu t))}} \\
&= \nu\sqrt{\lim_{t\to+\infty}\frac{F(g(t))}{(g(t))^{p+1}}  
\frac{(g(\nu t))^{p+1}}{F(g(\nu t))} \Big[\frac{g(t)}{t^{\frac{2}{k+2}}}\Big]^{p+1}               
\Big[\frac{(\nu t)^{\frac{2}{k+2}}}{g(\nu t)}\Big]^{p+1} \nu^{-\frac{2(p+1)}{k+2}}} \\
&= \nu^{\frac{k+1-p}{k+2}}\sqrt{\lim_{t\to+\infty}\frac{F(g(t))}{(g(t))^{p+1}}  
\lim_{t\to+\infty}\frac{(g(\nu t))^{p+1}}{F(g(\nu t))} \lim_{t\to+\infty}\Big[\frac{g(t)}{t^{\frac{2}{k+2}}}\Big]^{p+1}               
\lim_{t\to+\infty}\Big[\frac{(\nu t)^{\frac{2}{k+2}}}{g(\nu t)}\Big]^{p+1}} \\
&= \frac{1}{\nu^{\frac{p-k-1}{k+2}}}>1, \,\,\quad\text{being $\nu<1$ and $p>k+1$.}
\end{align*}
Let us now check that condition~\eqref{secondacurv} is fulfilled. Observe first that, using \eqref{caso:potpot}, \eqref{1.8} and ~\eqref{limetaa}, we have
\begin{align*}
& \lim_{\delta\to 0^+}\frac{f(g(\eta(\delta(1+o(1)))))}{f(g(\eta(\delta)))} \\
&=
\lim_{\delta\to 0^+}\frac{f(g(\eta(\delta(1+o(1)))))}{(g(\eta(\delta(1+o(1)))))^p}\cdot
\lim_{\delta\to 0^+}\frac{(g(\eta(\delta)))^p}  {f(g(\eta(\delta)))}\cdot \\
& \,\,\,\,\cdot \lim_{\delta\to 0^+}\Big[\frac{g(\eta(\delta(1+o(1))))}{(\eta(\delta(1+o(1))))^{\frac{2}{k+2}}}\Big]^p
\cdot \lim_{\delta\to 0^+}\Big[\frac{(\eta(\delta))^{\frac{2}{k+2}}}{g(\eta(\delta))}\Big]^p
\cdot \lim_{\delta\to 0^+}\Big[\frac{\eta(\delta(1+o(1)))}{\eta(\delta)}\Big]^{\frac{2p}{k+2}}  \\
&= \lim_{\delta\to 0^+}\Big[\frac{\eta(\delta(1+o(1)))}{\eta(\delta)}\Big]^{\frac{2p}{k+2}} \\
& =
\Big[\lim_{\delta\to 0^+}\frac{\eta(\delta(1+o(1)))}{(\delta(1+o(1)))^{\frac{k+2}{k+1-p}}}  
\lim_{\delta\to 0^+}\frac{\delta^{\frac{k+2}{k+1-p}}}{\eta(\delta)}
\lim_{\delta\to 0^+}\frac{(\delta(1+o(1)))^{\frac{k+2}{k+1-p}}}{\delta^{\frac{k+2}{k+1-p}}}
\Big]^{\frac{2p}{k+2}}=1.
\end{align*}
Moreover, we have
\begin{align*}
& \Big(\lim_{\delta\to 0^+}\frac{g'(\eta(\delta(1+o(1))))}{g'(\eta(\delta))}\Big)^2
=\lim_{\delta\to 0^+}\frac{a(g(\eta(\delta)))}{a(g(\eta(\delta(1+o(1)))))} \\
&=\lim_{\delta\to 0^+}\frac{a(g(\eta(\delta)))}{(g(\eta(\delta)))^k}\cdot
\lim_{\delta\to 0^+}\frac{(g(\eta(\delta(1+o(1)))))^k}{a(g(\eta(\delta(1+o(1)))))}\cdot
\lim_{\delta\to 0^+}\Big[\frac{g(\eta(\delta))}{(\eta(\delta))^{\frac{2}{k+2}}}\Big]^k \cdot \\
&\,\, \cdot \lim_{\delta\to 0^+}\Big[\frac{(\eta(\delta(1+o(1))))^{\frac{2}{k+2}} }{g(\eta(\delta(1+o(1))))}\Big]^k\cdot
\lim_{\delta\to 0^+}\Big[\frac{\eta(\delta)}{\delta^{\frac{k+2}{k+1-p}}}\Big]^{\frac{2k}{k+2}}\cdot \\
&\,\, \cdot\lim_{\delta\to 0^+}\Big[\frac{(\delta(1+o(1)))^{\frac{k+2}{k+1-p}}}{\eta(\delta(1+o(1)))}\Big]^{\frac{2k}{k+2}}\cdot
\lim_{\delta\to 0^+}\frac{\delta^{\frac{2k}{k+1-p}}}{(\delta(1+o(1)))^{\frac{2k}{k+1-p}}}=1.
\end{align*}
Arguing in a similar fashion, there holds
$$
\lim_{\delta\to 0^+}\frac{F(\eta(\delta))}{F(\eta(\delta(1+o(1))))}=1.
$$
Therefore, collecting the above conclusions, from the definition of $B$ it follows that
\begin{align*}
 \lim_{\delta\to 0^+}\frac{B(\eta(\delta(1+o(1))))}{B(\eta(\delta))}&=
\lim_{\delta\to 0^+}\frac{f(g(\eta(\delta(1+o(1)))))}{f(g(\eta(\delta)))}\cdot \\
& \,\,\cdot\lim_{\delta\to 0^+}\frac{g'(\eta(\delta(1+o(1))))}{g'(\eta(\delta))}\cdot
\sqrt{\lim_{\delta\to 0^+}\frac{F(\eta(\delta))}{F(\eta(\delta(1+o(1))))}}=1,
\end{align*}
as desired. 
For what concerns the quantity $B(w)\Lambda(w)$ we have, using \eqref{caso:potpot} again,
\begin{align}\label{f2}
\lim_{w\to +\infty}B(w) \Lambda(w)&=\lim_{w\to +\infty}\frac{f(g(w))g'(w)}{\sqrt{2}}\,\frac{\int_0^w \sqrt{2F(g(s))}ds}{(F(g(w)))^{3/2}}\nonumber \\
& =\lim_{w\to +\infty}\frac{ f_{\infty}(g(w))^p}{\sqrt 2 \sqrt {a(g(w))}}\,\frac{\sqrt{(p+1)^3}}{\sqrt{f_{\infty}^3}}
\frac{\int_0^w \sqrt{2F(g(s))}ds} { (g(w))^{\frac{3(p+1)}2}}\nonumber\\
&= \frac {\sqrt{(p+1)^3}}{\sqrt 2\sqrt{a_{\infty}} \sqrt f_{\infty}}\lim_{w\to +\infty}\frac{\int_0^w \sqrt{2F(g(s))}ds} { (g(w))^{\frac{p+3+k}2}}\nonumber\\
&= \frac {\sqrt{(p+1)^3}}{\sqrt 2\sqrt{a_{\infty}} \sqrt f_{\infty}}\,\frac 2{(p+k+3)}\lim_{w\to +\infty}\frac{ \sqrt{2F(g(w))}} { (g(w))^{\frac{p+1+k}2}g'(w)}\nonumber\\
&=\frac {2\sqrt{(p+1)^3}}{\sqrt{a_{\infty}} \sqrt f_{\infty} (p+k+3)}\lim_{w\to +\infty}
\frac {\sqrt f_{\infty}\sqrt{a_{\infty}}} {\sqrt {p+1}}\,\frac{(g(w))^{\frac{p+1}2}} { (g(w))^{\frac{p+1+k}2}(g(w))^{-\frac k2}}\nonumber\\
&=\frac{ 2(p+1)}{p+k+3}
\end{align}
so that condition~\eqref{terzacurv} follows from \eqref{f2}.\\
In turn, from Proposition \ref{fin1}, again on account of Lemma~\ref{l2.3} and by \eqref{limetaa}, up to possibly enlarging $L$ we obtain
\begin{align*}
& |u(x)-g\circ\eta({\rm d}(x,\partial\Omega)-{\mathcal H}(\sigma(x))J({\rm d}(x,\partial\Omega)))| \\ 
&\leq L\frac{({\rm d}(x,\partial\Omega))^{\frac{k+2}{k+1-p}}o({\rm d}(x,\partial\Omega))}{\min\{u^{k/2}(x)),
({\rm d}(x,\partial\Omega)-{\mathcal H}(\sigma(x))J({\rm d}(x,\partial\Omega)))^{k/(k+1-p)}\}} , 
\end{align*}
provided that ${\rm d}(x,\partial\Omega)$ goes to zero. This proves (1) and the first assertion of Theorem~\ref{blowII}.
Let us now come to the proof of assertions (2) and (3). 
First we want to estimate the function $J(t)$ near $0$. To this end, observe 
that the function $ \Lambda(t)$ is defined for $t>0$ and that, from \eqref{caso:potpot} it follows
\begin{align*}
& \lim_{t\to 0^+} \Lambda(\eta(t))=\lim_{w\to +\infty} \Lambda(w)= \lim_{w\to +\infty}\frac{\int_{0}^w\sqrt{2F(g(s))}ds}{F(g(w))}\\
& =\sqrt 2\lim_{w\to +\infty}\frac{\sqrt{F(g(w))}}{f(g(w))g'(w)}=\sqrt 2\lim_{w\to +\infty}\frac{\left( \frac {f_{\infty}}{p+1}\right)^{\frac 12}\left( g(w)\right)^{\frac{p+1}2}\sqrt{a(g(w))}}{f_{\infty}\left( g(w)\right)^p}\\
&= \frac{\sqrt 2}{\sqrt f_{\infty} \sqrt{p+1}}\lim_{w\to +\infty} \sqrt{a_{\infty}}(g(w))^{\frac{p+1}2 -p+\frac k2}=\frac{\sqrt 2\sqrt{a_{\infty}}}{\sqrt f_{\infty} \sqrt{p+1}}\lim_{w\to +\infty}(g(w))^{\frac{k+1-p}2}=0
\end{align*}
since $p>k+1$. This implies that
\begin{align}\label{f1}
& \lim_{t\to 0^+} J'(t)=\lim_{t\to 0^+}\frac{N-1}2 \Lambda(\eta(t))=0.
\end{align}
Moreover
\begin{align*}
\Lambda '(w)&=\frac{\sqrt {2 F(g(w))}F(g(w))-f(g(w))g'(w)\int_{0}^w\sqrt{2F(g(s))}ds}{\left(F(g(w))\right)^2}\\
& =\frac {\sqrt 2}{\sqrt{F(g(w))}}-\frac{f(g(w))g'(w)}{F(g(w))}\frac{\int_{0}^w\sqrt{2F(g(s))}ds}{F(g(w))}\\
&=\frac {\sqrt 2}{\sqrt{F(g(w))}}-\frac{f(g(w))g'(w)}{F(g(w))}  \Lambda(w),
\end{align*}
so that
\begin{align*}
J''(t)&= \frac{N-1}2 \Lambda '(\eta(t))\eta'(t)=-\frac {N-1}2 \Lambda '(\eta(t))\sqrt{2F(g(\eta(t)))}\\
&=-\frac {N-1}{\sqrt 2}\left( \sqrt 2- \frac{f(g(\eta(t)))g'(\eta(t))}{\sqrt{F(g(\eta(t)))}} \Lambda(\eta(t))\right)\\
&=(N-1)\big( -1+B(\eta(t))\Lambda(\eta(t))\big).
\end{align*}
Then  \eqref{f2} implies that
\begin{align}\label{f3}
\lim_{t\to 0^+}J''(t)&=(N-1)\left( -1+\lim_{t\to 0^+}B(\eta(t))  \Lambda(\eta(t))\right)=(N-1)\left( -1+\lim_{w\to +\infty}B(w)
 \Lambda(w)\right)\nonumber\\
&=(N-1)\left( -1+\frac{ 2(p+1)}{(p+k+3)}\right)=\frac{(N-1)}{(p+k+3)}(p-k-1).
\end{align}
From \eqref{f1} and \eqref{f3} we get
\begin{align}\label{f4}
J(t) &= \frac{(N-1)}{p+k+3}(p-k-1)t^2+o(t^2)
\end{align}
for $t>0$ sufficiently small.
From formulas \eqref{limetaa} and \eqref{f4} we have
\begin{align*}
& \eta({\rm d}(x,\partial\Omega)-{\mathcal H}(\sigma(x))J({\rm d}(x,\partial\Omega)))=\\
&=\frac{\eta({\rm d}(x,\partial\Omega)-{\mathcal H}(\sigma(x))J({\rm d}(x,\partial\Omega)))}
{\left({\rm d}(x,\partial\Omega)-{\mathcal H}(\sigma(x))J({\rm d}(x,\partial\Omega))\right)^{\frac{k+2}{k+1-p}}} 
\left({\rm d}(x,\partial\Omega)\right)^{\frac{k+2}{k+1-p}}\Big(1-\frac{{\mathcal H}(\sigma(x))J({\rm d}(x,\partial\Omega))}
{{\rm d}(x,\partial\Omega)}\Big)^{\frac{k+2}{k+1-p}}\\
&= \Gamma_0 (1+o(1))\left({\rm d}(x,\partial\Omega)\right)^{\frac{k+2}{k+1-p}}
\Big( 1+\frac{k+2}{p-k-1}{\mathcal H}(\sigma(x))\frac{J({\rm d}(x,\partial\Omega)}{{\rm d}(x,\partial\Omega)}
+ o\Big( \frac{J({\rm d}(x,\partial\Omega)}{{\rm d}(x,\partial\Omega)}\Big)\Big)\\
&= \Gamma_0({\rm d}(x,\partial\Omega))^{\frac{k+2}{k+1-p}}
\Big( 1+\frac{k+2}{p-k-1}{\mathcal H}(\sigma(x))\frac{(N-1)(p-k-1) }{p+k+3}
\frac{({\rm d}(x,\partial\Omega))^2}{ {\rm   d}(x,\partial\Omega)} (1+o(1))\Big) (1+o(1))
\end{align*}
if ${\rm d}(x,\partial\Omega)$ is sufficiently small. Moreover
\begin{align*}
 \eta({\rm d}(x,\partial\Omega))&o({\rm d}(x,\partial\Omega))=\\
&=\frac{\eta({\rm d}(x,\partial\Omega))}{\left( {\rm d}(x,\partial\Omega)\right)^{\frac{k+2}{k+1-p}}}
\left( {\rm d}(x,\partial\Omega)\right)^{\frac{k+2}{k+1-p}}o({\rm d}(x,\partial\Omega))\\
&=\Gamma_0 (1+o(1))\left( {\rm d}(x,\partial\Omega)\right)^{\frac{k+2}{k+1-p}+1}o(1)\\
&=\Gamma_0\left( {\rm d}(x,\partial\Omega)\right)^{\frac{2k+3-p}{k+1-p}}o(1)
\end{align*}
for ${\rm d}(x,\partial\Omega)$ small enough. Then \eqref{vsvilup} implies that
\begin{align*}
v(x)= \Gamma_0\left({\rm d}(x,\partial\Omega)\right)^{\frac{k+2}{k+1-p}}(1+o(1)) 
+\Gamma_0 \frac{(k+2)(N-1)}{p+k+3}{\mathcal H}(\sigma(x)) \left({\rm d}(x,\partial\Omega)\right)^{\frac{2k+3-p}{k+1-p}}(1+o(1))
\end{align*}
for ${\rm d}(x,\partial\Omega)$ small enough.
Then, in light of Lemma~\ref{colleg} and using \eqref{1.8}, 
any blow-up solution $u\in C^2(\Omega)$ to \eqref{prob} satisfies
\begin{align*}
& u(x)= g(v(x))\\
&= g\Big( \Gamma_0\left({\rm d}(x,\partial\Omega)\right)^{\frac{k+2}{k+1-p}}(1+o(1)) 
+\Gamma_0 \frac{(k+2)(N-1)}{p+k+3}{\mathcal H}(\sigma(x)) \left({\rm d}(x,\partial\Omega)\right)^{\frac{2k+3-p}{k+1-p}}(1+o(1))\Big)\\
&= \frac{g\left( \Gamma_0\left({\rm d}(x,\partial\Omega)\right)^{\frac{k+2}{k+1-p}}(1+o(1)) 
+\Gamma_0 \frac{(k+2)(N-1)}{p+k+3}{\mathcal H}(\sigma(x)) \left({\rm d}(x,\partial\Omega)\right)^{\frac{2k+3-p}{k+1-p}}(1+o(1))\right)}
{\left[  \Gamma_0\left({\rm d}(x,\partial\Omega)\right)^{\frac{k+2}{k+1-p}}(1+o(1)) 
+\Gamma_0 \frac{(k+2)(N-1)}{p+k+3}{\mathcal H}(\sigma(x)) \left({\rm d}(x,\partial\Omega)\right)^{\frac{2k+3-p}{k+1-p}}(1+o(1))\right]^{\frac 2{k+2}}}\cdot \\
& \cdot \Big[  \Gamma_0\left({\rm d}(x,\partial\Omega)\right)^{\frac{k+2}{k+1-p}}(1+o(1)) 
+\Gamma_0 \frac{(k+2)(N-1)}{p+k+3}{\mathcal H}(\sigma(x)) \left({\rm d}(x,\partial\Omega)\right)^{\frac{2k+3-p}{k+1-p}}(1+o(1))\Big]^{\frac 2{k+2}}\\
&=g_{\infty}(1+o(1))\Gamma_0^{\frac 2{k+2}}\left({\rm d}(x,\partial\Omega)\right)^{\frac 2{k+1-p}}
\left( 1+\frac{(k+2)(N-1)}{p+k+3}{\mathcal H}(\sigma(x)) {\rm d}(x,\partial\Omega)(1+o(1))\right)^{\frac 2{k+2}}\\
&=g_{\infty}\Gamma_0^{\frac 2{k+2}}\left({\rm d}(x,\partial\Omega)\right)^{\frac 2{k+1-p}}(1+o(1))
+\frac {2(N-1)}{p+k+3}g_{\infty}\Gamma _0^{\frac 2{k+2}}{\mathcal H}(\sigma(x))\left( {\rm d}(x,\partial\Omega)\right)^{\frac{k+3-p}{k+1-p}}(1+o(1))
\end{align*}
as $x$ approaches the boundary $\partial\Omega$. By performing some simple manipulations this yield the desired conclusions.

\end{proof}

\noindent
{\bf Acknowledgments.} The authors wish to thank Professors Catherine Bandle, Louis Dupaigne, Alberto Farina, Olivier Goubet 
and Vicentiu Radulescu (in alphabetical order) for providing some bibliographic data. 
The authors also thank Marco Caliari
for providing a numerical Octave code in order to check the validity of some asymptotic expansions.

\end{document}